\documentclass{amsart}
\usepackage{float}% Force la position des tableaux
\usepackage{amsmath,amsfonts,amssymb,amsthm}
\usepackage{enumerate}
\usepackage[utf8]{inputenc}
\usepackage[T1]{fontenc}
\usepackage[english]{babel}
\usepackage{graphicx}
\usepackage{color}
\usepackage{cases}
\usepackage{epstopdf}
\usepackage[colorlinks=true]{hyperref}
\hypersetup{urlcolor=blue, citecolor=blue}

\usepackage{pgf,tikz}
\usepackage{mathrsfs}
\usetikzlibrary{arrows}
\usepackage{subcaption}
\usepackage{ulem}
\usepackage{color}
\newtheorem{deff}{Definition}[section]
\newtheorem{prop}[deff]{Proposition}
\newtheorem{thm}[deff]{Theorem}
\newtheorem{lem}[deff]{Lemma}

\newtheorem{rmk}[deff]{\it Remark}
\newcounter{cst}
\newcommand*\diff{\mathop{}\!\mathrm{d}}
\def\R{\mathbb{R}}
\def\N{\mathbb{N}}

\def\E{\mathcal{E}}

\def\T{\mathcal{T}}
\def\K{\mathcal{K}}

\def\Ee{\mathfrak{E}}
\def\O{\Omega}

\def\p{\partial}

\def\grad{\nabla}
\def\d{{\rm d}}
\def\1{{\bf 1}}

\def\be{\begin{equation}}
	\def\ee{\end{equation}}

\def\Nn{\mathcal{N}}
\def\longrightharpoonup{{-\!\!\!\rightharpoonup}}

\newcommand{\raye}[1]{}

\begin{document}
\title[Large time behavior of a two phase PME]{Large time behavior of a two phase extension of the porous medium equation}
\thanks{This work was supported by the French National Research Agency through grant 
ANR-13-JS01-0007-01 (project GEOPOR)}

\author[ Ait Hammou Oulhaj]{Ahmed Ait Hammou Oulhaj}\address{Ahmed Ait Hammou Oulhaj:  Inria, Univ. Lille, CNRS, UMR 8524 - Laboratoire Paul Painlev\'e, F-59000 Lille 
(\href{mailto:ahmed.ait-hammou-oulhaj@math.univ-lille1.fr}{\tt ahmed.ait-hammou-oulhaj@math.univ-lille1.fr})}
\author[ Canc\`es]{Cl\'ement Canc\`es}\address{Cl\'ement Canc\`es:  Inria, Univ. Lille, CNRS, UMR 8524 - Laboratoire Paul Painlev\'e, F-59000 Lille (\href{mailto:clement.cances@inria.fr}{\tt clement.cances@inria.fr})} 
\author[Chainais-Hillairet]{Claire Chainais-Hillairet}\address{Claire Chainais-Hillairet: 
Univ. Lille,  CNRS, UMR 8524, Inria - Laboratoire Paul Painlev\'e, F-59000 Lille (\href{mailto:claire.chanais@math.univ-lille1.fr}{\tt claire.chanais@math.univ-lille1.fr})} 
\author[ Laurençot]{Philippe Laurençot}\address{Philippe Lauren\c{c}ot: Institut de Math\'ematiques de Toulouse, UMR 5219, 
Universit\'e de Toulouse, CNRS, F-31062 Toulouse Cedex 9, 
(\href{mailto:laurenco@math.univ-toulouse.fr}{\tt laurenco@math.univ-toulouse.fr})}

\begin{abstract}
We study the large time behavior of the solutions to a two phase extension of the porous medium equation, which models the so-called seawater intrusion problem. The goal is to identify the self-similar solutions that correspond to steady states of a rescaled version of the problem.  We fully characterize the unique steady states that are identified as minimizers of a convex energy and shown to be radially symmetric.
Moreover, we prove the convergence of the solution to the time-dependent model towards the unique stationary state as time goes to infinity. We finally provide numerical illustrations of the stationary states and we exhibit numerical convergence rates.
\end{abstract}

\maketitle

\noindent
{\small {\bf Keywords.}
	Two-phase porous media flows, Muskat problem, large time behavior, cross diffusion
	\vspace{5pt}
	
	\noindent
	{\bf AMS subjects classification. }
	35K65, 35K45, 76S05
}
\section{Introduction}
\subsection{Presentation of the continuous problem}

The purpose of this work is to investigate the large time behavior of a seawater intrusion model 
which is a two-phase generalization of the porous medium equation (PME). 
The model we are interested in is derived by Jazar and Monneau in \cite{Jazar}, 
where the authors consider the Dupuit approximation of an unsaturated immiscible two-phase 
(freshwater and saltwater) within an unconfined aquifer assuming that the interface 
between both fluids is sharp (the fluids occupy disjoint regions), see also \cite{Escher,WoMa00} for alternative derivations of the same model. 
This yields a 2D reduced model obtained from a full 3D model where the unknowns are 
the heights of the fluid layers. More precisely the interface between 
the saltwater and the bedrock is set at $\{z = 0\}$, whereas the height of the freshwater (resp. saltwater) 
layer is denoted by $\{z={f}(t,x)\}$ (resp. $\{z={g}(t,x)\}$), see Figure~\ref{aquifer}.
The model proposed in~\cite{Jazar} is 
\begin{equation}
\begin{cases}
\partial_t {f} -\nabla{\cdot}\big(\nu {f}\nabla({f}+{g})\big)=0 \quad &\text{in}\quad (0,\infty)\times\R^2, \\
\partial_t {g}-\nabla{\cdot}\big({g}\nabla(\rho {f}+{g})\big)=0 \quad&\text{in}\quad (0,\infty)\times\R^2,\\
{f}_{|t=0}={f}_0,\quad {g}_{|t=0}={g}_0 \quad  &\text{in} \quad \R^2,
\end{cases}
\label{system}
\end{equation}
where $\rho=\dfrac{\rho_{\text{fresh}}}{\rho_{\text{salt}}} \in (0,1)$ is the density ratio between the two fluids, 
and where the parameter $\nu= \dfrac{\nu_{\text{salt}}}{\nu_{\text{fresh}}} >0$ is the ratio of the kinematic viscosities.
%%%%%%%%%%%%%%%%%%%%%%%%%%%%%%%%%%%%%%%%%%%%%%%%%%%%%%%%%%%%%%%%%%%
%Figure 
%%%%%%%%%%%%%%%%%%%%%%%%%%%%%%%%%%%%%%%%%%%%%%%%%%%%%%%%%%%%%%%%%%%%%
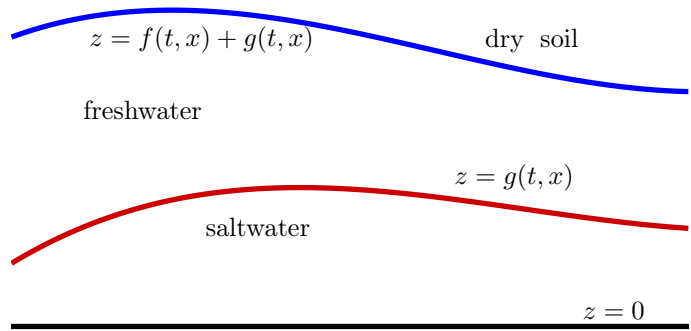
\begin{figure}[htb] 
	\begin{center} 
\definecolor{xdxdff}{rgb}{0.49,0.49,1.}
\definecolor{uuuuuu}{rgb}{0.27,0.27,0.27}

\definecolor{qqqqff}{rgb}{0.,0.,1.}
\definecolor{ccqqqq}{rgb}{0.8,0.,0.}
\begin{tikzpicture}[scale=1.5][line cap=round,line join=round,>=triangle 45,x=1.0cm,y=1.0cm]
\clip(3.,0.9) rectangle (9.,4.);
\draw[line width=2.pt,color=ccqqqq,smooth,samples=100,domain=3.0:9.0] plot(\x,{0.012086387182830342*(\x)^(3.0)-0.273455383599484*(\x)^(2.0)+1.9191785423052505*(\x)-2.062573339296535});
\draw[line width=2.pt,color=qqqqff,smooth,samples=100,domain=3.0:9.0] plot(\x,{0.01395571987851084*(\x)^(3.0)-0.28356314944941396*(\x)^(2.0)+1.688694212676448*(\x)+0.6759212725394361});
\draw [line width=2.pt] (0.,0.)-- (13.,0.);
\draw (7.115532536786527,3.724257902533276) node[anchor=north west] {{dry  \;soil}};
\draw (3.5552646598002586,3.0849552090992507) node[anchor=north west] {freshwater};
\draw (4.638235877134485,2.0425954124150585) node[anchor=north west] {{saltwater}};
\draw (6.844789732452971,2.5194524232489496) node[anchor=north west] {${z={g}(t,x)}$};
\draw (3.60941322066697,3.743035061233209) node[anchor=north west] {${z={f} (t,x)+{g}(t,x)}$};
\draw [line width=2.pt] (0.,1.)-- (9.132566429071524,1.0002356157308658);
\draw (7.981909510653908,1.3011098037479453) node[anchor=north west] {${z=0}$};
\begin{scriptsize}
\draw[color=ccqqqq] (-0.5464888258531233,-3.4196406650144437) node {$g$};
\draw[color=qqqqff] (-0.5464888258531233,-0.5497669390787451) node {$h$};
\draw [fill=uuuuuu] (0.,0.) circle (1.5pt);
\draw[color=uuuuuu] (0.08975676433073451,0.19477577283853525) node {$P$};
\draw [fill=xdxdff] (13.,0.) circle (2.5pt);
\draw[color=xdxdff] (13.098948512558128,0.24892433370524653) node {$Q$};
\draw [fill=xdxdff] (0.,1.) circle (2.5pt);
\draw[color=xdxdff] (0.08975676433073451,1.2506727097394055) node {$R$};
\draw [fill=qqqqff] (9.132566429071524,1.0002356157308658) circle (2.5pt);
\end{scriptsize}
\end{tikzpicture}
		\caption{The physical setting}
		\label{aquifer}
	\end{center}
\end{figure}
%%%%%%%%%%%%%%%%%%%%%%%%%%%%%%%%%%%%%%%%%%%%%%%%%%%%%%%%%%%%%%
%%%%%%%%%%%%%%%%%%%%%%%%%%%%%%%%%%%%%%%%%%%%%%%%%%%%%%%%%%%%%%%%%%%%%%
The authors in \cite{Escher,Escher2} studied the classical solutions of system (\ref{system}). 
Moreover, the existence of weak solutions is established under different assumptions in \cite{Escher3,laurencot,laurencot2,Li}.

The characteristic time corresponding to the aquifer dynamics is large. Therefore, understanding the large-time behavior of system~(\ref{system}) is of great interest. The so-called entropy method \cite{AMTU01,Jungel_entropy2016} provides a powerful approach to study the long time behavior of  different systems of PDEs. It has been developed first for kinetic equations (Boltzmann and Landau \cite{Toscani_Villani2000}). Then it was extended to other problems, as the linear Fokker-Planck equation \cite{Carillo_Toscani1998}, the porous medium equation (PME) \cite{Carillo_Toscani2000, Vazquez2007}, reaction-diffusion systems \cite{Desvillettes2006, Desvillettes2015, Glitzky}, drift-diffusion systems for semiconductor devices \cite{Gajewski1996, Gajewski1986, Gajewski1989}, and thin film models \cite{Carlen2005}.
For more details about this method and its application domains, one can refer to \cite{AMTU01, Jungel_entropy2016} and the references therein.
Similar results were obtained in \cite{Bolley2012, Bolley2013,
 Otto2001, ZM15}  based on the interpretation of the PDE models as the gradient flow of a certain energy functional with respect to the Wasserstein metric. We refer for instance to the monographs \cite{Ambrosio2008,Santambrogio_OTAM} for an extensive discussion on this topic.

In \cite{Laurencot_CompoLong}, Laurençot and Matioc studied the large-time behavior of the system (\ref{system}) in the one-dimensional case.  In their paper a classification of self-similar solutions is first provided : there is always a unique even self-similar solution (also found in \cite{WoMa00}) while a continuum of non-symmetric self similar solutions exists for certain fluid configurations. The authors proved the convergence of all nonnegative weak solutions towards a self-similar solution. Nevertheless nothing is known about the rate of convergence. Surprisingly, the situation is simpler in the 2D case, as we shall see below.

As already mentioned, the system (\ref{system})  can be interpreted as a two-phase generalization of the PME. 
This can be easily seen by choosing $f\equiv0$ or $g\equiv 0$ in the system~\eqref{system}. 
In order to explain the principles of the entropy method, let us consider the following PME 
\begin{equation}
\begin{cases}
\partial_t v=\Delta v^2\quad &\text{in}\;(0,\infty)\times \R^2,\\
v(0,x)=v_0(x)\geq 0\quad &\text{in}\; \R^2.
\end{cases}
\label{eq_mil_por_LargeTime}
\end{equation}
A further transformation of (\ref{eq_mil_por_LargeTime}) involves the so-called self-similar variables (see \cite{Carillo_Toscani2000, Vazquez2007})  and reads
\be u(t,x)=e^{2t}v\Big(\dfrac{1}{4}(e^{4t}-1),xe^t\Big),\qquad (t,x)\in [0,\infty)\times\R^2, \label{auto_sim_PME} \ee
which transforms (\ref{eq_mil_por_LargeTime}) into the nonlinear Fokker-Planck equation
\begin{equation}
\begin{cases}
\partial_t u=\text{div} (xu+\grad u^2)\quad &\text{in}\;(0,\infty)\times \R^2,\\
u(0,x)=v_0(x)\geq 0\quad &\text{in}\; \R^2.
\end{cases}
\label{Fok_Plank_PME}
\end{equation}
In \cite{Carillo_Toscani2000} the authors study the large time behavior of the PME (\ref{eq_mil_por_LargeTime}) using the rescaled equation (\ref{Fok_Plank_PME}). The energy corresponding to (\ref{Fok_Plank_PME}) is 
\be\label{eq:NRJ-PME}
H(u)=\int_{\R^2} \Big(|x|^2u+2u^2\Big)\diff x.
\ee
Given a nonnegative initial condition $v_0\in L^1(\R^2)$ with $M:=\|v_0\|_{L^1(\R^2)}>0$, they prove that the unique stationary solution of (\ref{Fok_Plank_PME}), which is given by
the Barenblatt-Pattle type formula
\be\label{eq:Barenblatt}
u_\infty(x)= \max\Bigg\{ \beta-\dfrac{1}{4}|x|^2 , 0 \Bigg\}, \qquad x\in\R^2,
\ee
where $\beta$ is determined by $\| u_\infty\|_{L^1(\R^2)} = M$ attracts the corresponding solution $u$ to \eqref{Fok_Plank_PME} at an exponential rate.
Moreover $u_\infty$  is the unique minimizer of $H$ in 
$$ \Big\{u\in L^2(\R^2)\cap L^1\big(\R^2,(1+|x|^2)\diff x\big)\;\; : \;\; \|u\|_{L^1(\R^2)}=M \Big\}.
$$
More precisely, the relative entropy of $u$ with respect to $u_\infty$ is then defined by
$$H(u|u_\infty):=H(u)-H(u_\infty)\geq 0,$$
whereas the entropy production for $H(u|u_\infty)$  is given by
$$I(u)=\int_{\R^2} u\big|x+2\grad u\big|^2 \diff x.$$
Let the initial data $v_0$ satisfy $v_0\in L^1(\R^2)\cap L^\infty (\R^2)$ with $|x|^{2+\delta} v_0\in L^1(\R^2)$ for some $\delta >0$. Then the corresponding solution $u$ to \eqref{Fok_Plank_PME} satisfies
$$
\lim_{t\to\infty} H (u(t)|u_\infty)=0 \quad  \text{and}\quad \lim_{t\to\infty} I(u(t)) = 0.
$$
Moreover, $H (u(t)|u_\infty)$ and $I(u(t))$ are linked by the relation
\be \dfrac{\diff}{\diff t}H (u(t)|u_\infty)=-2 I(u(t)),\label{entrR_dissipPME}\ee
and
\be \dfrac{\diff}{\diff t}I(u(t))=-2I(u(t))-R(t),\label{dissip_eqPME}\ee
where $R(t)\geq 0$. Combining (\ref{entrR_dissipPME}) and (\ref{dissip_eqPME}) one obtains
\be 
\dfrac{\diff}{\diff t}H (u(t)|u_\infty)=\dfrac{\diff}{\diff t}I(u(t))+R(t) \ge \dfrac{\diff}{\diff t}I(u(t)).\label{rel_dtH_dtI}
\ee
Integrating (\ref{rel_dtH_dtI}) with respect to $t$ over $(0,\infty)$ gives
\be 0\leq H(u(t)|u_\infty) \leq I(u(t)),\quad t>0.\label{ineg_entrR_dissipPME}\ee
Substituting (\ref{ineg_entrR_dissipPME}) into (\ref{entrR_dissipPME}), one concludes with the exponential decay of the relative entropy to zero at a rate 2.

The goal of this paper is to apply a similar strategy to describe the long time behavior of the system~(\ref{system}). Our approach relies also in the use of self-similar variables~(\ref{self_similar_var}), leading to the introduction of quadratic confining  potentials. The advantage of this alternative formulation is that the profiles of nonnegative self-similar solutions to (\ref{system}) are
nonnegative stationary solutions to (\ref{system2}). We will give an explicit characterization of the self-similar profiles in Section~\ref{self_Profiles} and, relying on compactness arguments, we prove the convergence towards a stationary solution. Unfortunately, due to the extended complexity of the problem~(\ref{system}) with respect to (\ref{eq_mil_por_LargeTime}) we are not able to establish the exponential convergence towards a steady state. This motivates the numerical investigation carried out in Section~\ref{numer_invest}, using a Finite Volume scheme~\cite{intrusion, moifvca2} which preserves at the discrete level the main features of the continuous problem (in particular the nonnegativity of the solutions, conservation of mass and decay of energy).

The outline of the paper is as follows. In the next section we
state the main results of our paper. 
As a preliminary step, we introduce a rescaled
version (\ref{system2}) of the  system (\ref{system}) which relies in particular on the introduction of self-similar variables. In Theorem~\ref{thm:exisUniq_profil} we state the existence and uniqueness of nonnegative stationary solutions to (\ref{system2}), which are moreover radially symmetric, compactly supported and Lipschitz continuous. The convergence of any nonnegative weak solution to (\ref{system2}) towards these stationary solutions is stated in Theorem~\ref{thm:conv}.
In Section~\ref{ppte_pbCont}  we prove Theorem~\ref{thm:exisUniq_profil} and Theorem~\ref{thm:conv}. In Section~\ref{self_Profiles} we give a classification of the self-similar profiles and we exhibit critical values of the parameter $\nu$ for which the shape of the stationary profile changes. We finally present in Section~\ref{numer_invest} numerical simulations for different values of $\nu$ in order to observe the stationary solutions  and the decay of the  relative energy.

%%%%%%%%%%%%%%%%%%%%%%%%%%%%%%%%%%%%%%%%%%%%%%%%%%%%%%%
%   Section main results
%%%%%%%%%%%%%%%%%%%%%%%%%%%%%%%%%%%%%%%%%%%%%%%%%%%%%
\section{Main results}
In what follows, given $M>0$, we use the following closed convex set
\begin{equation*} 
\K_{M}:=\Big\{h\in L^2(\R^2)\cap L^1\big(\R^2,(1+|x|^2)\diff x\big):
\ h\geq 0\;\; \text{a.e. and}\;\; 
\| h\|_{L^1(\R^2)} =M \Big\}.
\end{equation*}

\subsection{Self-similar solutions}
The main contribution of this paper is the classification of nonnegative self-similar solutions to (\ref{system}), that is, solutions of the form 
\be 
({f},{g})(t,x)=(1+t)^{-1/2}(F,G)(x (1+t)^{-1/4}),\quad (t,x)\in (0,\infty)\times \R^2.
\label{changt_autoSimil}
\ee
Deeper insight on this issue is provided by transforming \eqref{system} using the so-called self-similar variables \cite{Carillo_Toscani2000, Vazquez2007}, i.e.,
\be
 ({f},{g})(t,x)=\dfrac{1}{(1+t)^{1/2}}(\tilde f,\tilde g)\Big(\log(1+t),\dfrac{x}{(1+t)^{1/4}}\Big).\label{self_similar_var}
 \ee
We end up with the following rescaled system
\begin{equation}
\begin{cases}
\partial_t \tilde f -\nabla{\cdot}\Big(\nu \tilde f\nabla(\tilde f+\tilde g+b/\nu)\Big)=0 \quad &\text{in}\quad(0,\infty)\times\R^2, \\
\partial_t \tilde g-\nabla{\cdot}\Big(\tilde g\nabla(\rho \tilde f+\tilde g+b)\Big)=0 \quad &\text{in}\quad (0,\infty)\times\R^2,\\
\tilde f_{|t=0}=f_0,\quad \tilde g_{|t=0}=g_0 \quad  &\text{in} \quad \R^2,
\end{cases}
\label{system2}
\end{equation}
where $b(x)=|x|^2/8$. Thus this change of variable preserves the nature of the equations but adds a 
confining drift term, the confining potentials $b/\nu$ and $b$ being different as soon as the two phases have different kinematic viscosities, i.e., $\nu \neq 1$.
The resulting system~\eqref{system2} still has a gradient flow structure, but for a modified energy in comparison to~\eqref{system}. 
Indeed the system (\ref{system2}) can be interpreted as the gradient flow with respect to the $2-$Wasserstein 
metric~\cite{laurencot} of the following energy
\be \Ee (\tilde f, \tilde g)=\int_{\R^2} E(\tilde f,\tilde g)\diff x, \ee
where 
\be 
E(\tilde f,\tilde g)= \dfrac{\rho}{2}(\tilde f+\tilde g)^2+\dfrac{1-\rho}{2}\tilde g^2+b\Big(\dfrac{\rho}{\nu}\tilde f+\tilde g\Big).
\label{E_NRJ}\ee

A corner stone of our study is that, if $(f,g)$ is a self-similar solution \eqref{changt_autoSimil} to \eqref{system}, then the corresponding self-similar profile $(F,G)$ is a stationary solution to (\ref{system2}). We will see in what follows that, given $M_f>0$ and $M_g>0$, there is a unique nonnegative stationary solution $(F,G)$ to \eqref{system2} satisfying $\|F\|_{L^1(\R^2)}=M_f$ and $\|G\|_{L^1(\R^2)}=M_g$ and it is the unique minimizer of $\Ee$ in $\K_{M_f}\times \K_{M_g}$. Furthermore, it satisfies the system
\begin{equation*}
\begin{cases}
F\nabla\phi_F&=0 \quad\text{in}\quad\R^2, \\
G\nabla\phi_G&=0 \quad\text{in}\quad \R^2,
\end{cases}
\end{equation*}
where we introduce the potentials
\begin{align*}
\phi_F := F + G + \frac{b}{\nu},
\qquad 
\phi_G := \rho F + G + b.
\end{align*}
In other words, the fluxes expressed in the self-similar variables are identically equal to zero. 

In what follows, we mainly work on the system~\eqref{system2} expressed in self-similar variables. 
In order to lighten the notations, we remove the tildes on $\tilde f$ and $\tilde g$ and denote solutions to~\eqref{system2} by $(f,g)$ while the steady states to~\eqref{system2} are denoted by $(F,G)$. 

Given two positive real numbers $M_f>0$ and $M_g>0$ and a stationary solution $(F,G)\in \K_{M_f}\times \K_{M_g}$ to (\ref{system2}), we define the positivity sets $E_F$ and $E_G$ of $F$ and $G$ by
$$ E_F=\{x\in \R^2 \; :\; F(x)>0\},\quad E_G=\{x\in \R^2\; :\; G(x)>0\},$$
and notice that $E_F$ and $E_G$ are both nonempty as
\be \|F\|_{L^1(\R^2)} = \int_{\R^2} F\diff x=M_f>0 \quad  \text{and}\quad \|G\|_{L^1(\R^2)} = \int_{\R^2}G\diff x=M_g>0.\label{M_f_M_gposi}\ee

\subsection{Main results}

Let $(F,G)\in \K_{M_f}\times \K_{M_g}$ be a stationary solution of (\ref{system2}). Let $M_f>0$ and $M_g>0$ be two positive real numbers.

\begin{thm}[Self-similar profiles]
There exists a unique stationary solution $(F,G)\in \K_{M_f}\times \K_{M_g}$ of (\ref{system2}) such that $F$ and $G$ belong to $H^1(\R^2)$ and $\sqrt{F}\phi_F$ and $\sqrt{G}\phi_G$ belong to $L^2(\R^2)$. It is radially symmetric, compactly supported, and Lipschitz continuous, and satisfies
\begin{equation*}
\begin{cases}
F\nabla(F+G+b/\nu)&=0 \quad\text{a.e. in}\quad\R^2, \\
G\nabla(\rho F+G+b)&=0 \quad\text{a.e. in}\quad \R^2.
\end{cases}
\end{equation*}
Moreover, $E_F$ and $E_G$ are bounded connected sets and  $$ (F,G)\in \underset{(f,g)\in \K_{M_f} \times \K_{M_g}}{\operatorname{arg\,min}} \Ee(f,g).$$
\label{thm:exisUniq_profil}
\end{thm}
An important result of this paper is the complete classification of the self-similar profiles $(F,G)$. In fact, in contrast with the 1D case~\cite{Laurencot_CompoLong}, we prove that all self-similar profiles $(F,G)$ are radially symmetric and thus can be computed explicitly, see Section~\ref{self_Profiles}. Besides being of interest to compare the outcome of numerical simulations with the theoretical predictions, this feature is at the heart of the uniqueness proof. Still, as it will be explained in Section~\ref{self_Profiles}, the shape of $F$ and $G$ strongly depends on $\nu$, $\rho$, and the mass ratio 
${M_f}/{M_g}$. The topology of $E_F$ and $E_G$ changes according to the values of these parameters.

The second result concerns the convergence of weak solutions to \eqref{system2} towards the stationary state.
\begin{thm}[Convergence towards the stationary state]
Let $(f_0,g_0)\in \K_{M_f}\times \K_{M_g}$ and consider the weak solution $(f,g)$ of (\ref{system2}) given by Theorem~\ref{thm:exists}. Then
$$\big(f(t), g(t)\big) \longrightarrow (F, G) \;\text{ in }\; L^2(\R^2;\R^2) \;\text{ as }\; t \rightarrow \infty,$$
where $(F,G)$ is the unique stationary solution of \eqref{system2} in $\K_{M_f}\times \K_{M_g}$ given by Theorem~\ref{thm:exisUniq_profil}.
\label{thm:conv}
\end{thm}
Theorem~\ref{thm:conv} guarantees the convergence of the solutions of (\ref{system2}) to the steady state but provides no information on the  rate of convergence. 
We conclude the paper by a numerical investigation concerning the convergence speed. 
The situation appears to be more intricate than for the (rescaled) single phase PME \eqref{Fok_Plank_PME} and, even though exponential convergence 
is always observed in our numerical tests, the rate strongly varies with the data and goes close to zero for some values of $\nu$. 
It is in particular rather unclear if there exists a uniform strictly positive minimum decay rate at which convergence occurs. 

%%%%%%%%%%%%%%%%%%%%%%%%%%%%%%%%%%%%%%%%%%%%%%%%%%%%%%%%%%
% Section the steady states and convergence
%%%%%%%%%%%%%%%%%%%%%%%%%%%%%%%%%%%%%%%%%%%%%%%%%%%%%%%%%%

\section{The steady states and convergence towards a minimizer}\label{ppte_pbCont}

We fix $M_f>0$ and $M_g>0$.

\subsection{Energy/energy dissipation for weak solutions}
Let $\omega \in L^1_{\rm loc}(\R^2;\R_+)$ with $\mathbb{R}_+=(0,\infty)$. For $r\ge 1$, we define $L^r_\omega(\R^2)$ as the set of measurable functions $\psi$ such that 
$$
\int_{\R^2} |\psi(y)|^r \omega(y) \diff y < \infty,
$$
which is a Banach space once equipped with the norm 
$$
\|\psi\|_{L^r_\omega(\R^2)} = \left(\int_{\R^2} |\psi(y)|^r \omega(y) \diff y\right)^{1/r}.
$$
For further use, we define the phase potentials 
\begin{equation}
\phi_f=f+g+b/\nu \;\text{ and }\; \phi_g=\rho f+g+b. \label{phasepot}
\end{equation}

\begin{deff}[weak solution]\label{defws}
Consider $(f_0,g_0)\in \K_{M_f}\times \K_{M_g}$. A pair $(f,g): \R^2\times\R_+ \to \R_+^2$ is a weak solution to the problem (\ref{system2}) if 
\begin{enumerate}[(i)]
\item $f$ and $g$ belong to $L^\infty(\R_+;L^2(\R^2))$ and to $L^\infty(\R_+;L^1_{|x|^2}(\R^2))$, 
\item $\grad f$ and $\grad g$ belong to $L^2_{\rm loc}(\R_+;L^2(\R^2))$,
\item $\sqrt{f}\nabla\phi_f$ and $\sqrt{g}\nabla\phi_g$ belong to $L^2(\R_+ \times \R^2)$
\item for all $\xi \in C^\infty_c(\R_+ \times \R^2)$, there holds 
$$
\int_{\R_+} \int_{\R^2} f \p_t \xi \diff x  \diff t + \int_{\R^2} f_0 \xi(0,\cdot) \diff x - \int_{\R_+} \int_{\R^2} \nu f \grad \phi_f \cdot \grad \xi \diff x \diff t= 0, 
$$
$$
\int_{\R_+} \int_{\R^2} g \p_t \xi \diff x  \d t + \int_{\R^2} g_0 \xi(0,\cdot) \diff x -  \int_{\R_+} \int_{\R^2} g \grad \phi_g \cdot \grad \xi \diff x \diff t= 0.
$$
\end{enumerate}
\end{deff}
The existence of a weak solution of the problem (\ref{system2}) and the decay of the energy $\Ee$ are given by the following theorem.
\begin{thm}\label{thm:exists}
Consider $(f_0,g_0)\in \K_{M_f}\times \K_{M_g}$. There exists a weak solution in the sense of Definition~\ref{defws}. Moreover, it satisfies the energy inequality
\be\label{eq:NRJ}
\Ee(f,g)(t) + \int_s^t \mathcal{I}(f,g)(\tau)\diff \tau \leq \Ee(f,g)(s), \quad t \ge s \ge 0,
\ee
with the energy dissipation $\mathcal I$ given by
 \be 
\mathcal{I}(f,g)=\int_{\R^2} \left(\nu \rho f |\grad \phi_f|^2 + g |\grad \phi_g|^2\right) \diff x,\label{eq:dissip}
\ee
and the entropy inequality
\begin{multline}\label{eq:entropy}
\mathcal{H}(f,g)(t) + \int_s^t \mathcal{D}(f,g)(\tau) \diff\tau \\ \le  \mathcal{H}(f,g)(s) 
+ \left( \frac{\rho M_f}{2} + \frac{\nu M_g}{2} \right)(t-s), \quad t \ge s \ge 0,
\end{multline}
where
\begin{align*}
\mathcal{H}(f,g)&=\int_{\R^2}\left(\rho f\ln{f}+ \nu g\ln{g}\right) \diff x, \\
\mathcal{D}(f,g)&=\nu\rho \|\nabla(f+g)\|_{L^2(\R^2)}^2+\nu(1-\rho) \|\nabla g\|_{L^2(\R^2)}^2.
\end{align*}
\end{thm}
Since $u \log(u) \leq u^2$ for $u \geq 0$, one has 
\be\label{eq:entropy_up}
\mathcal{H}(f,g) (t)\leq\rho\|f(t)\|_{L^2(\R^2)}^2 + \nu\|g(t)\|_{L^2(\R^2)}^2, \quad \text{for a.e.}\;t \in \R_+.
\ee
On the other hand, by~\cite[Eq. (2.14)]{Arkeryd72},
\be\label{eq:entropy_low}
\mathcal{H}(f,g)(t) \geq - \pi(\rho+\nu)+\rho \|f(t)\|_{L^1_{|x|^2}(\R^2)} + \nu \|g\|_{L^1_{|x|^2}(\R^2)}, 
\quad \text{for a.e.}\;t \in \R_+.
\ee
Combining~\eqref{eq:entropy_up} and \eqref{eq:entropy_low}, one gets that $\mathcal{H}(f,g)$ belongs to 
$L^\infty(\R_+)$, so that~\eqref{eq:entropy} makes sense for almost every $t\geq s\geq 0$.

The existence of a weak solution was proven  by Lauren\c{c}ot and Matioc in \cite{laurencot,laurencot2} by proving the convergence of a JKO scheme without the confining potentials $b/\nu$ and $b$, but the proof can be extended in the presence of these quadratic confining potentials without particular difficulties. The $L_{\rm loc}^2(L^2)$ estimates on $\grad f$ and $\grad g$ and the entropy inequality \eqref{eq:entropy} are obtained thanks to 
the flow interchange technique of Matthes \textit{et al.} \cite{MMS09}. 

An important consequence of Theorem~\ref{thm:exists} is that, if $(F,G)\in \K_{M_f}\times \K_{M_g}$ is a stationary solution to \eqref{system2} such that
\begin{equation}
(F,G)\in H^1(\R^2;\R^2) \;\text{ and }\; (\sqrt{F}\grad\phi_F,\sqrt{G}\grad\phi_G)\in L^2(\R^2;\R^4), \label{regstatsol}
\end{equation} 
then we infer from \eqref{eq:NRJ} and the nonnegativity of $\mathcal{I}$ that $\mathcal{I}(F,G)=0$ (recall that $\phi_F$ and $\phi_G$ are defined by \eqref{phasepot}). In other words, $F$ and $G$ have vanishing fluxes and
\begin{equation}
\begin{cases}
F\nabla(F+G+b/\nu)&=0 \quad\text{in}\quad\R^2, \\
G\nabla(\rho F+G+b)&=0 \quad\text{in}\quad \R^2.
\end{cases}
\label{system_stat}
\end{equation}

\subsection{Existence and properties of the minimizer of the energy}

As already mentioned, the problem \eqref{system} can be interpreted as a two phase generalization of the PME \eqref{eq_mil_por_LargeTime}, and its long time behavior is expected to share some common features with that equation. In particular, since the Barenblatt profile $u_\infty$ given by~\eqref{eq:Barenblatt} is the unique minimizer of the energy 
functional~\eqref{eq:NRJ-PME} under a mass constraint, we are led to consider the following minimization problem
\be  \inf_{(f,g)\in \K_{M_f}\times \K_{M_g}} \Ee (f, g). \label{pb_min}\ee 
Owing to the energy inequality \eqref{eq:NRJ}, a minimizer $(F,G)$ of $\Ee$ in $\K_{M_f}\times \K_{M_g}$ is obviously a stationary solution to \eqref{system2} and thus satisfies \eqref{system_stat}.

In order to prove the uniqueness of the minimizer in~\eqref{pb_min}, 
we need the strict convexity of the energy functional $\Ee$.
\begin{prop}
$\Ee$ is a strictly convex function on $\K_{M_f}\times \K_{M_g}$.
\label{srictconvex}
\end{prop}
\begin{proof}
If $E$ is strictly convex then $\Ee$ is strictly convex. We denote by $D^2 E$ the Hessian matrix of $E$. Then 
$$
D^2 E=
\begin{pmatrix} 
\rho & \rho \\
\rho & 1
\end{pmatrix}.$$ 
Recalling that $\rho\in (0,1)$, the matrix $D^2 E$ is symmetric with $\text{det}(D^2 E)=\rho (1-\rho)>0$ and $\text{tr}(D^2 E)=1+\rho>0.$ 
We deduce that $D^2E$ is definite positive and  hence $E$ is strictly convex.
\end{proof}

Let us now introduce some material that will be needed to prove the existence of a minimizer of~\eqref{pb_min}. There exist $C_\star>0$ and $C^\star>0$ depending only on $\rho$ and $ \nu$ such that
\be  C_\star (\Nn(f)+\Nn(g))\leq \Ee(f, g)\leq C^\star (\Nn(f)+\Nn(g)) , \label{controlnormeL2} \ee 
where $$\Nn(h)=\|h\|^2_{L^2(\R^2)}+\|h\|_{L^1_{|x|^2}(\R^2)},\quad \text{for}\quad h\in L^2(\R^2)\cap L^1_{|x|^2}(\R^2).$$
Indeed, on the one hand, since $f$, $g$ and $b$ are nonnegative, there holds 
\begin{align*}
E(f,g)&=\dfrac{\rho}{2}(f+g)^2+\dfrac{1-\rho}{2}g^2+b(\dfrac{\rho}{\nu}f+g)\\&
 \geq \dfrac{\rho}{2}(f^2+\dfrac{1}{\nu}bf)+\dfrac{1}{2}(g^2+bg).
\end{align*}
On the other hand, using $(u+v)^2\leq 2(u^2 + v^2)$, we have
$$E(f,g)\leq \rho f^2+\dfrac{1+\rho}{2}g^2+\dfrac{\rho}{\nu}bf+bg,$$
which implies (\ref{controlnormeL2}).
This motivates the introduction of the Banach space 
\[
X = L^2(\R^2) \cap L^1_{|x|^2}(\R^2) \qquad \text{with}\quad \|\cdot\|_X = \|\cdot\|_{L^2} + \|\cdot\|_{L^1_{|x|^2}}.
\]
We say that a sequence $\left(u_n\right)_{n\geq 1}$ converges in $X$ in the weak-$\star$ sense towards $u$ if $u_n$ converges to $u$ weakly in $L^2(\R^2)$ and if the densities of the moments of order~$2$ $x\mapsto u_n(x)|x|^2$ converge weakly in the sense of finite measures (i.e., in the dual of the space $C_0(\R^2)$ of the continuous functions 
decaying to $0$ as $|x|\to\infty$) towards $x\mapsto u(x)|x|^2$. 
Then any bounded sequence in $X$ is relatively compact in $X$ for the weak-$\star$ topology.

We can now go to the following statement. 
\begin{prop}\label{prop:minimizer}
There exists a unique minimizer $(F, G)$ of $\Ee$ in $\K_{M_f}\times \K_{M_g}$. Moreover, $F$ and $G$ belong to $H^1(\R^2)$ while $\sqrt{F}\phi_F$ and $\sqrt{G}\phi_G$ belong to $L^2(\R^2)$ and it satisfies  \eqref{system_stat}, the fluxes $\phi_F$ and $\phi_G$ being defined in \eqref{phasepot}.
\end{prop}
\begin{proof}
The uniqueness of the minimizer follows from the strict convexity of the energy functional $\Ee$ proved in Proposition \ref{srictconvex}.

Let us now prove the existence of a minimizer. To this end, pick a minimizing sequence $(f_k, g_k)_{k\ge 1}\in \K_{M_f}\times \K_{M_g}$. Thanks to (\ref{controlnormeL2}) there exists a constant $C>0$ such that 
\be \|f_k\|_{X}+\|g_k\|_{X}\le C, \; \forall k\ge 1. \label{suite_min}\ee
We obtain that there exist $(F, G)\in L^2(\R^2; \R_+)^2$ and a subsequence of  $(f_k, g_k)_{k\ge 1}$ (denoted again by  $(f_k, g_k)_{k\ge 1}$) such that $$ f_k \to F \;\; \text{weakly-$\star$ in}\;\;X \quad  \text{and}\; \;g_k \to G \;\; \text{weakly-$\star$ in}\;\; X.$$
This convergence implies in particular that $\|F\|_{L^1(\R^2)}=M_f$ and $\|G\|_{L^1(\R^2)}=M_g$, hence $(F,G)\in \K_{M_f}\times \K_{M_g}$.

Moreover, the energy functional $\Ee$ is lower semi-continuous for the weak-$\star$ topology of $X$. Thus,
$$ \Ee(F, G)\le \underset{k \to\infty}{\rm liminf}\; \Ee(f_k, g_k),$$ 
so that $(F, G)$ is a minimizer of $\Ee$ in $\K_{M_f}\times \K_{M_g}$. 

Let us now show that~\eqref{system_stat} holds. To this end, define $(\check{f},\check{g})$ as a solution of the evolutionary system
\begin{equation}
\begin{cases}
\partial_t \check{f} -\nabla{\cdot}\Big(\nu \check{f}\nabla(\check{f}+\check{g}+b/\nu)\Big)=0 \quad &\text{in}\quad(0,\infty)\times\R^2, \\
\partial_t \check{g}-\nabla{\cdot}\Big(\check{g}\nabla(\rho \check{f}+\check{g}+b)\Big)=0 \quad &\text{in}\quad (0,\infty)\times\R^2,\\
\check{f}_{|t=0}=F,\quad \check{g}_{|t=0}=G \quad  &\text{in} \quad \R^2.
\end{cases}
\label{system3}
\end{equation}
 Using \eqref{eq:NRJ} one has
$$
\Ee(\check{f},\check{g})(t) + \int_0^t \mathcal I(\check{f},\check{g})(\tau)\diff \tau \leq \Ee(F,G), 
$$
with 
$$ 
\mathcal{I}(\check{f},\check{g})=\int_{\R^2} \left(\nu \rho\check{f} |\grad \phi_{\check{f}}|^2 + \check{g} |\grad \phi_{\check{g}}|^2\right) \diff x,
$$ 
where $\phi_{\check{f}}=\check{f}+\check{g}+b/\nu$ and $\phi_{\check{g}}=\rho \check{f}+\check{g}+b$. 
Since $(F,G)$ is a minimizer of $\Ee$ in $\K_{M_f}\times \K_{M_g}$ and $(\check{f}(t),\check{g}(t))$ belongs to $\K_{M_f}\times \K_{M_g}$ for all $t\ge 0$, we deduce from the nonnegativity of  $\mathcal{I}$ that $\Ee(\check{f}(t),\check{g}(t))=\Ee(F,G)$ and $\mathcal{I}(\check{f}(t),\check{g}(t))=0$ for a.e. $t>0$. Owing to the minimizing property of $(F,G)$, the first identity readily implies that $(\check{f}(t), \check{g}(t))=(F,G)$ for all $t\ge 0$. On the one hand, it follows from Theorem~\ref{thm:exists} that $(F,G)$ enjoys the regularity properties listed in Proposition~\ref{prop:minimizer}. On the other hand, we conclude that $(F,G)$ satisfies $\mathcal{I}(F,G)=0$ and thus \eqref{system_stat}. 
\end{proof}

With Proposition~\ref{prop:minimizer} at hand, we are now interested in the regularity of stationary solutions $(F,G)$ to \eqref{system2} in $\K_{M_f}\times \K_{M_g}$ and in the description of the positivity sets $E_F$ and $E_G$ of $F$ and $G$ defined by
\[ 
E_F=\{x\in \R^2 \; :\; F(x)>0\},\quad E_G=\{x\in \R^2\; :\; G(x)>0\}.
\]

\begin{lem}\label{lem:radial}
Let $(F,G)\in \K_{M_f}\times \K_{M_g}$ be a solution of \eqref{system_stat} enjoying the regularity properties \eqref{regstatsol}. Then
\begin{enumerate}[(i)]
\item $ E_F\cap E_G\neq \emptyset$, 
\item $F$ and $G$ are locally Lipschitz continuous and radially symmetric.
\end{enumerate}
\end{lem}
\begin{proof}
(i) Assume for contradiction that $E_F\cap E_G= \emptyset$. Then $F\nabla G=G\nabla F=0$ in $\R^2$ and it follows from \eqref{system_stat} that   $(F,G)$  satisfy the equations 
\begin{equation*}
\begin{cases}
F\nabla(F+b/\nu)&=0 \quad\text{in}\quad\R^2, \\
G\nabla(G+b)&=0 \quad\text{in}\quad \R^2,
\end{cases}
\end{equation*}
hence $F$ and $G$ are Barenblatt solutions centred at $0$. Thus $0\in E_F\cap E_G=\emptyset$, yielding a contradiction.

\noindent(ii) Let us prove that $(F,G)$ are locally Lipschitz continuous and radially symmetric. We have
$$ \R^2= (E_F\cap E_G)\cup(E_F\cap E_G^c)\cup (E_F^c\cap E_G)\cup (E_F^c \cap E_G^c).$$
Since $F$ and $G$ belong to $H^1(\R^2)$ by \eqref{regstatsol}, it follows from Stampacchia's theorem that
$$ \nabla F=0 \; \text{a.e. in}\; E_F^c \quad \text{and}\;\; \nabla G=0 \; \text{a.e. in} \; E_G^c.$$
Therefore 
\be 
\nabla G=-\nabla b\;\; \text{a.e in}\;\; E_F^c\cap E_G,\;
\nabla F=-\nabla b/\nu \;\; \text{a.e on}\;\; E_G^c\cap E_F, \label{F_G_sur_compl}\ee
\be 
\nabla F=\dfrac{\nu-1}{\nu(1-\rho)}\nabla b \;\; \text{and} \;\; \nabla G=\dfrac{\rho-\nu}{\nu(1-\rho)}\nabla b\;\; \text{a.e in} \;\; E_F\cap E_G.
\label{F_G_sur_intersec}\ee 
Since $\nabla b(x)=x/4$ is locally bounded in $\R^2$, both $F$ and $G$ are locally Lipschitz continuous in $\R^2$. In addition, $\nabla F(x)\cdot x^\bot=\nabla G(x)\cdot x^\bot=0$ for almost all $x=(x_1,x_2)\in\R^2$ with $x^\bot=(-x_2,x_1)$. Consequently, $F$ and $G$ are radially symmetric. 
\end{proof}

According to the discussion above the profiles $(F,G)$ of self-similar solutions of (\ref{system}) defined in (\ref{changt_autoSimil}) are stationary solutions of \eqref{system2} and  satisfy \eqref{regstatsol} and \eqref{system_stat}. Moreover $(F,G)$ are radially symmetric. Then we can express $F$ and $G$ as functions of $r = |x|$. Thanks to (\ref{F_G_sur_compl})--(\ref{F_G_sur_intersec}) one has:
\begin{enumerate}[$\bullet$]
\item On $E_F\cap E_G$, there are $(C_1, C_2)\in \R^2$ such that 
\be F(r)=C_1+\dfrac{\nu-1}{8\nu(1-\rho)}r^2,\qquad G(r)= C_2+ \dfrac{\rho-\nu}{8\nu(1-\rho)}r^2. \label{OnE_FandE_G} \ee 
\item On $E_F\cap E_G ^c$, there is $C_3 \in \R$ such that
\be F(r)=C_3-\dfrac{1}{8\nu}r^2,\qquad G(r)=0.\label{OnE_F}\ee
\item On $E_G\cap E_F^c$, there is $C_4\in \R$ such that
\be G(r)=C_4-\dfrac{1}{8}r^2, \qquad F(r)=0. \label{OnE_G}\ee
\end{enumerate}
The above statements are actually not completely correct  since the quantities $C_1, C_2, C_3$ and $C_4$ are constant only on the connected components of $E_F\cap E_G$, $E_F\cap E_G ^c$, and $E_G\cap E_F^c$, respectively. But this ambiguity will be removed thanks to the following lemma. 
\begin{lem}\label{connected_sets}
Let $(F,G)\in \K_{M_f}\times \K_{M_g}$ be a solution of \eqref{system_stat} enjoying the regularity properties \eqref{regstatsol}. Then $0\in E_F \cup E_G$ and $E_F$ and $E_G$ are connected sets and bounded.
\end{lem}
\begin{proof} \textbf{Step~1.}
Assume first for contradiction that $E_F\cap E_G$ has an unbounded connected component $\mathcal{O}$. On $\mathcal{O}$, $(F,G)$ are given by \eqref{OnE_FandE_G} and the integrability of $F$ and $G$ complies with \eqref{OnE_FandE_G} only when $\rho>\nu>1$, hence a contradiction.

Assume next for contradiction that $E_F$ has an unbounded connected component. Since we have already proved that the connected components of $E_F\cap E_G$ are bounded, the radial symmetry of $F$ implies that there is $r_0>0$ such that $\{x\in\R^2\ :\ |x|>r_0\}\subset E_F\cap E_G^c$. But this contradicts \eqref{OnE_F}. A similar argument excludes unbounded connected components in $E_G$.

\noindent\textbf{Step~2.} Thanks to formulas (\ref{OnE_FandE_G})--(\ref{OnE_G}), $F$ is nonincreasing (as a function of $r$) when $\nu\in (0,1]$ while $G$ is nonincreasing when $\nu>1>\rho$. We only consider when $F$ is nonincreasing, the other case being handled similarly. Owing to the monotonicity of $F$ and the positivity of $\|F\|_{L^1(\R^2)}=M_f$, it is obvious that $0\in E_F$ and $E_F$ is a disk centred at $x=0$, hence connected and bounded thanks to the previous step. 

Assume for contradiction that $E_G$ is not connected and let $\mathcal{O}_1=\{ x\in\R^2\ :\ r_1<|x|<r_2\}$ and $\mathcal{O}_2=\{ x\in\R^2\ :\ r_3<|x|<r_4\}$ be two connected components of $E_G$ with $0\le r_1<r_2\le r_3<r_4$. On the one hand, $G$ is increasing in a right-neighborhood of $\{ x\in\R^2\ :\ r_3=|x|\}$ which is only possible if this neighborhood is included in $E_F$ and $\rho\ge\nu$ according to formulas 
\eqref{OnE_FandE_G}--\eqref{OnE_G}. Since $r_3\ge r_2$, this fact and the monotonicity of $F$ imply that $F(r)\ge F(r_3)>0$ as soon as $r\le r_3$, hence $\mathcal{O}_1\subset E_F$. Since $\rho\ge\nu$, $G$ is increasing on $\mathcal{O}_1$ by \eqref{OnE_FandE_G} and thus cannot vanish on $\{ x\in\R^2\ :\ r_2=|x|\}$, hence a contradiction. Therefore, $E_G$ is also connected and its boundedness follows also from the previous step.
\end{proof}

As a consequence of Lemmas~\ref{lem:radial} and~\ref{connected_sets} we get that  $F$ and $G$ are compactly supported and globally Lipschitz continuous.
%%%%%%%%%%%%%%%%%%%%%%%%%%%%%%%%%%%%%%%%%%%%%%%%%%%%%%%%%%%%%%
%%%%%%%%%%%%%%%%%%%%%%%%%%%%%%%%%%%%%%%%%%%%%%%%%%%%%%%%%%%%%%%%

%=======================================
%          Convergence
%=======================================
\subsection{Convergence towards a minimizer}
The goal of this section is to make a step towards Theorem~\ref{thm:conv} where the convergence as $t\to\infty$ of a solution $(f(t),g(t))$ to~\eqref{system2} with initial conditions $(f_0,g_0)\in \K_{M_f}\times \K_{M_g}$ towards the unique stationary solution $(F,G)\in \K_{M_f}\times \K_{M_g}$ is proved. The first part of the proof consists in showing by compactness arguments that any cluster point of $(f(t),g(t))$ as $t\to\infty$ is a stationary solution to \eqref{system2} satisfying \eqref{regstatsol}. 

\begin{proof}[Proof of Theorem~\ref{thm:conv}]
Let $n \in \N$. We define $(f_n, g_n): \R_+\times \R^2 \to (\R_+)^2$ by 
$$
f_n(t,x) = f(t+n,x), \qquad g_n(t,x) = g(t+n,x),
$$ 
and set $\Ee_n = \Ee(f_n,g_n)(0) = \Ee(f,g)(n)$. The relation \eqref{eq:NRJ} yields
$$
\Ee_{n+1} + \int_0^1 \int_{\R^2}  \Psi_n \diff x\diff t \leq \Ee_n, \qquad n \ge 0,
$$
where we have set 
$$
\Psi_n =\nu \rho f_n |\grad \phi_{f_n}|^2 + g_n |\grad \phi_{g_n}|^2 \in L^1((0,1)\times \R^2).
$$
Since $\Psi_n \ge 0$ and $\Ee\ge 0$, we deduce from the previous inequality that $\left(\Ee_n\right)_n$ is non-increasing and
\begin{equation}
0 \le \Ee_n + \int_0^1 \sum_{n\in\N} \Psi_n \diff t \le \Ee_0. \label{ZZZ}
\end{equation}
In particular,
\be 
\Psi_n \underset{n\to\infty}\longrightarrow 0 \quad \text{in}\; L^1((0,1)\times \R^2).
\label{conv_Psi}\ee
Thanks to \eqref{controlnormeL2} and \eqref{ZZZ}, we have
\begin{itemize}
\item $\left(f_n\right)_{n}$ and $\left(g_n\right)_{n}$ are bounded in $L^\infty((0,1);L^2(\R^2)),$
\item $\left(f_n\right)_{n}$ and $\left(g_n\right)_{n}$ are bounded in $L^\infty((0,1);L^1_{|x|^2}(\R^2)).$
\end{itemize}
In addition, it follows from \eqref{eq:entropy} and  the bounds~\eqref{eq:entropy_up}--\eqref{eq:entropy_low} 
on the entropy $\mathcal{H}$ that 
\begin{align*}
& \int_0^1 \left( \nu\rho \|\nabla(f_n+g_n)\|_{L^2(\R^2)}^2 + \nu(1-\rho) \|\nabla g_n\|_{L^2(\R^2)}^2 \right) \diff t \\
& \quad = \int_n^{n+1} \left( \nu\rho \|\nabla(f+g)\|_{L^2(\R^2)}^2 + \nu(1-\rho) \|\nabla g\|_{L^2(\R^2)}^2 \right) \diff t \\
& \quad \le \mathcal{H}(f,g)(n) - \mathcal{H}(f,g)(n+1) + \frac{\rho M_f + \nu M_g}{2} \leq C.
\end{align*}
Consequently,
\begin{itemize}
\item $\left(f_n\right)_{n}$ and $\left(g_n\right)_{n}$ are bounded in $L^2((0,1);H^1(\R^2)).$
\end{itemize}
Moreover, 
\begin{itemize}
\item $\left( \partial_t f_n \right)_n$ and $\left( \partial_t g_n \right)_n$ converge to zero in $L^2 \left( (0,1);(W^{1,4}(\R^2))' \right)$ as $n\to\infty$.
\end{itemize} 
Indeed, for $\varphi\in L^2((0,1);W^{1,4}(\R^2))$,
 \begin{align*}
 \left|\int_0^1 \langle \partial_t f_n, \varphi \rangle \diff t   \right|&=  \left|\int_0^1\int_{\R^2}\grad\cdot \left(\nu f_n\grad \phi_{f_n} \right)\varphi \diff x \diff t \right|\\ &\leq \int_0^1\int_{\R^2} \left|\nu  f_n\grad \phi_{f_n}\cdot \grad \varphi \right|\diff x \diff t\\  &\leq \|\nu f_n\grad\phi_{f_n}\|_{L^2((0,1);L^{4/3}(\R^2))} \| \varphi \|_{L^2((0,1);W^{1,4}(\R^2))},
\end{align*}
and 
\begin{align*}
 \|f_n\grad\phi_{f_n}\|^2_{L^2((0,1);L^{4/3}(\R^2))}&\leq  \int_0^1 \left\|\sqrt{f_n}\right\|^2_{L^4(\R^2)} \left\|\sqrt{f_n}\grad \phi_{f_n}\right\|^2_{L^2(\R^2)} \diff t\\&\leq C\int_0^1 \int_{\R^2} \Psi_n \diff x\diff t \mathop{\longrightarrow}_{n\to \infty} 0, 
 \end{align*}
thanks to (\ref{conv_Psi}). The proof for $(\partial_t g_n)_n$ is similar.
Since the embedding $H^1(\R^2)\cap L^1_{|x|^2}(\R^2)$ in $L^1(\R^2)\cap L^2(\R^2)$ is compact, see \cite[Lemma~A.1]{laurencot2}, and the embedding $L^1(\R^2)\cap L^2(\R^2)$ in $(W^{1,4}(\R^2))'$ is continuous, thanks to Lemma~\ref{embedding_L1L2_dual_W14}, we are in a position
to apply \cite[Corollary~4]{Sim87}  to conclude that there are a subsequence
of $(f_n, g_n)_n$ (not relabeled) and functions $(F,G)$ such that
\begin{align*} 
&(f_n, g_n)\longrightarrow (F, G) \quad \text{strongly in}\quad L^1((0,1)\times \R^2)\cap L^2((0,1)\times \R^2),\\
& (\grad f_n, \grad g_n) -\!\!\!\rightharpoonup (\grad F, \grad G) \quad \text{weakly in}\quad L^2((0,1)\times \R^2),\\
& (\partial_t f_n, \partial_t g_n)\longrightarrow (0,0) \quad \text{strongly in}\quad L^2((0,1);W^{1,4}(\R^2)').
\end{align*}
This implies in particular that $(\p_t F, \p_t G) = (0,0)$. Moreover, there exist $\chi_F, \chi_G$ in $L^2((0,1)\times\R^2)^2$
such that 
\begin{align}
&\sqrt{f_n} \grad \phi_{f_n}\longrightharpoonup \chi_F \quad \text{weakly in}\quad L^2((0,1)\times\R^2)^2, \label{eq:chi_F}\\
&\sqrt{g_n} \grad \phi_{g_n}\longrightharpoonup \chi_G \quad \text{weakly in}\quad L^2((0,1)\times\R^2)^2.\label{eq:chi_G}
\end{align}
But since $\sqrt{f_n}$ (resp. $\sqrt{g_n}$) converges strongly in $L^2((0,1)\times\R^2)$ towards $\sqrt F$ (resp. $\sqrt G$), 
and since $\grad \phi_{f_n}$ (resp. $\grad \phi_{g_n}$) converges weakly in $L^2((0,1); L^2_{\rm loc}(\R^2))$ towards 
$\grad \phi_F$ (resp. $\grad \phi_{G}$), we can identify the limits in~\eqref{eq:chi_F}--\eqref{eq:chi_G} as
\[
\chi_F = \sqrt{F}\grad \phi_F, \qquad \chi_G = \sqrt{G}\grad \phi_G.
\]
Owing to (\ref{conv_Psi}), we have moreover that
\begin{align*}
\sqrt{f_n}\grad \phi_{f_n} & \longrightarrow 0 \quad \text{strongly in}\quad L^2((0,1)\times \R^2), \\
\sqrt{g_n}\grad \phi_{g_n} & \longrightarrow 0 \quad \text{strongly in}\quad L^2((0,1)\times \R^2).
\end{align*}
Therefore, since $F$ and $G$ do not depend on time, $\sqrt{F}\grad \phi_{F}=\sqrt{G} \phi_G = 0$ a.e. in $\R^2$, that is, $(F,G)$ solves \eqref{system_stat}. Furthermore, 
\begin{equation*}
\int_{\R^2} F \diff x = \lim_{n\to\infty} \int_0^1 \int_{\R^2} f_n \diff x\diff t = M_f.
\end{equation*} 
A similar argument being available for $G$, we conclude that $(F,G)\in \K_{M_f} \times \K_{M_g}$. We have thus established that $(F,G)$ is a solution to \eqref{system_stat} in $\K_{M_f} \times \K_{M_g}$ which satisfies \eqref{regstatsol}.
 
It remains to check that the whole sequence $\left(f_n, g_n\right)_n$ converges towards $(F,G) \in \K_{M_f} \times \K_{M_g}$. This is a consequence of the uniqueness of the solution to the stationary problem~\eqref{system_stat}, {\it cf.} Remark~\ref{rmk:uniqueness} in Section~\ref{self_Profiles} below.
 \end{proof}
 
 %%%%%%%%%%%%%%%%%%%%%%%%%%%%%%%%%%%%%%%%%%%%%%%%%%%%%%%%%%%%%%%%%%%%%%%%
 %%%%%%%%%%%%%%%%%%%%%%%%%%%%%%%%%%%%%%%%%%%%%%%%%%%%%%%%%%%%%%%%%
\section{Explicit characterization of the self-similar profiles }\label{self_Profiles}
%%%%%%%%%%%%%%%%%%%%%%%%%%%%%%%%%%%%%%%%%%%%%%%%%%%%%%%%%%%%%%%
%%%%%%%%%%%%%%%%%%%%%%%%%%%%%%%%%%%%%%%%%%%%%%%%%%%%%%%%%%%%%%%%%%

The viscosity ratio $\nu$ appears to play a central role in the characterization of the stationary profiles. Therefore, we will suppose that $M_f>0$, $M_g>0$ and $\rho \in (0,1)$ are fixed and we classify the stationary solutions with respect to the values of $\nu$.  We define the critical values of $\nu$ by
\be \nu_1^\star=\dfrac{\rho^2M_f}{M_g+\rho(M_f-M_g)},\quad \nu_2^\star=\dfrac{\rho M_f+M_g}{M_f+M_g},\quad \nu_3^\star =1+(1-\rho)\dfrac{M_f}{M_g}. \label{val_critiq}\ee
It is easy to check that $$0<\nu_1^\star <\rho<\nu_2^\star <1<\nu_3^\star.$$
It follows from \eqref{OnE_FandE_G}-\eqref{OnE_G} and Lemmas~\ref{lem:radial} and~\ref{connected_sets} that only four configurations are possible for the stationary solutions. Figure~\ref{figure_soluStat} illustrate these four configurations. 
We now show that these four configurations correspond to the four intervals $(0, \nu_1^\star]$, $(\nu_1^\star, \nu_2^\star)$, $(\nu_2^\star, \nu_3^\star)$ and $[\nu_3^\star,\infty)$, some kind of degeneracy taking place when $\nu\in\{\nu_1^\star,\nu_2^\star,\nu_3^\star\}$. In the particular case $\nu  = \nu_2^\star$, one has $E_F = E_G$.
%%%%%%%%%%%%%%%%%%%%%%%%%%%%%%%%%%%%%%%%%%%%%%%%%%%%%%%%%%%%%%%
%              La figure des états stationnaires
%%%%%%%%%%%%%%%%%%%%%%%%%%%%%%%%%%%%%%%%%%%%%%%%%%%%%%%%%%%%
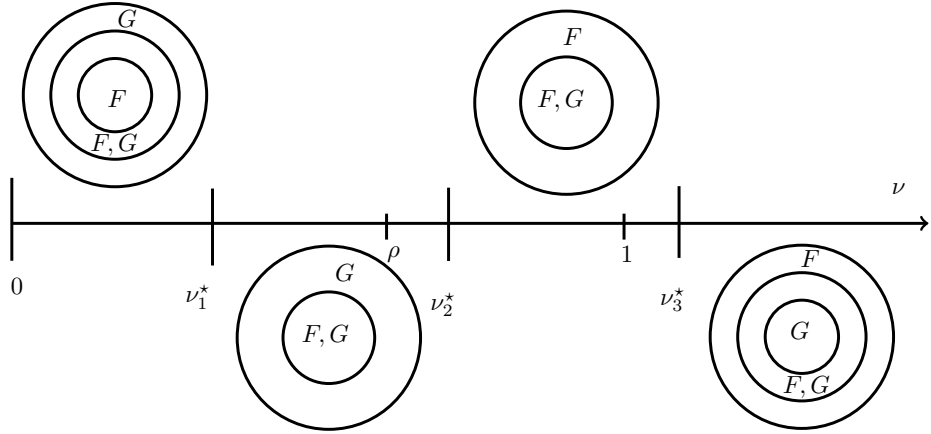
\begin{figure}[htbp]
\begin{center} 
\resizebox{\textwidth}{!}{
\begin{tikzpicture}[scale=1.3][line cap=round,line join=round,>=triangle 45,x=1.0cm,y=1.0cm]
\clip(-0.1,-2.3) rectangle (10.1,2.45);
\draw [line width=1.2pt] (2.1915382332828153,0.38)-- (2.1915382332828153,-0.46);

\draw [line width=1.2pt] (4.766114743820066,0.3896312668968353)-- (4.766114743820066,-0.41036873310316474);
\draw (0.95,1.5544698805971019) node[anchor=north west] {$F$};
\draw (0.7651046470255683,1.0668275387302113) node[anchor=north west] {$F,G$};
\draw (1.0610462538609788,2.4279808378436665) node[anchor=north west] {$G$};
\draw (3.0717616042400105,-1.019902409939972) node[anchor=north west] {$F,G$};
\draw (5.638199587087243,1.5465355729072647) node[anchor=north west] {$F,G$};
\draw (5.918090099903637,2.226061128444977) node[anchor=north west] {$F$};
\draw (8.396484287877426,-0.983924307376693) node[anchor=north west] {$G$};
\draw (8.3166302707889,-1.576177432625994) node[anchor=north west] {$F,G$};
\draw (8.508513484459721,-0.19842064927570188) node[anchor=north west] {$F$};
\draw (3.4356010230373895,-0.348311162092097) node[anchor=north west] {$G$};
\draw (9.503907655377105,0.5271560002810258) node[anchor=north west] {$\nu$};
\draw (6.98,-0.6121505808894765) node[anchor=north west] {$\nu_3^\star$};
\draw (4.4749076042078896,-0.6601213843071818) node[anchor=north west] {$\nu_2^\star$};
\draw (4.0,-0.1501213843071818) node[anchor=north west] {$\rho$};
\draw [line width=1.2pt] (4.089860913299552,0.1058946826221381)-- (4.089860913299552,-0.1771857952123998);
\draw (1.8,-0.581286834527555) node[anchor=north west] {$\nu_1^\star$};
\draw (-0.10224572901837384,-0.50635495696553897) node[anchor=north west] {$0$};

\draw [line width=1.2pt] (0.0029784765304113965,0.4967681822044667)-- (0.0029784765304113965,-0.4083766692161724);
\draw [->,line width=1.2pt] (0.0029784765304113965,0.) -- (10.,0.);

\draw [line width=1.2pt] (1.1280257214749543,1.3995659603922253) circle (1.cm);
\draw [line width=1.2pt] (1.1280257214749543,1.3995659603922253) circle (0.7cm);
\draw [line width=1.2pt] (1.1280257214749543,1.3995659603922253) circle (0.4cm);
\draw [line width=1.2pt] (3.4604883649494895,-1.248265099995376) circle (1.cm);
\draw [line width=1.2pt] (3.4604883649494895,-1.248265099995376) circle (0.5cm);
\draw [line width=1.2pt] (6.05166884629204,1.3169195221654622) circle (0.5cm);
\draw [line width=1.2pt] (6.05166884629204,1.3169195221654622) circle (1.cm);
\draw [line width=1.2pt] (8.619311463260026,-1.2382272921338742) circle (0.4cm);
\draw [line width=1.2pt] (8.619311463260026,-1.2382272921338742) circle (0.7cm);
\draw [line width=1.2pt] (8.619311463260026,-1.2382272921338742) circle (1.cm);

\draw [line width=1.2pt] (7.279860913299552,0.4058946826221381)-- (7.279860913299552,-0.3771857952123998);

\draw [line width=1.2pt] (6.679860913299552,0.1058946826221381)-- (6.679860913299552,-0.1771857952123998);
\draw (6.55,-0.15635495696553897) node[anchor=north west] {$1$};
\end{tikzpicture}
}
\caption{Shape of the stationary profiles according to the value of the parameter $\nu$}
\label{figure_soluStat}

	\end{center}
\end{figure}
%%%%%%%%%%%%%%%%%%%%%%%%%%%%%%%%%%%%%%%%%%%%%%%%%%%%%%%%%%%%%%%%%%%%%%%%%
%%%%%%%%%%%%%%%%%%%%%%%%%%%%%%%%%%%%%%%%%%%%%%%%%%%%%%%%%%%%%%%%%%%%%%%%%%%%
\paragraph{\bf First case :}  the first configuration we consider is shown on Figure~\ref{first_config} for some $0 < r_1 \leq r_2$. We note that  $0\in E_F\cap E_G$.  
%%%%%%%%%%%%%%%%%%%%%%%%%%%%%%%%%%%%%%%%%%%%%%%%%%%%%%%%%%%%%%%%%%%%%%%%%%%%%%
%
%%%%%%%%%%%%%%%%%%%%%%%%%%%%%%%%%%%%%%%%%%%%%%%%%%%%%%%%%%%%%%%%%%%%%%%%%%%%%%
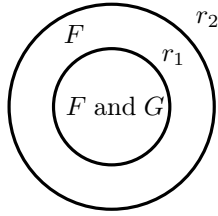
\begin{figure}[!htbp] 
	\begin{center} 

\begin{tikzpicture}[scale=0.9][line cap=round,line join=round,>=triangle 45,x=1.0cm,y=1.0cm]
\clip(-2.7,-2.1) rectangle (0.6,1.2);
\draw [line width=1.2pt] (-1.04,-0.44) circle (0.8590692637965812cm);
\draw [line width=1.2pt] (-1.04,-0.44) circle (1.5153877391611692cm);
\draw [line width=1.2pt] (3.02,-0.42) circle (0.8955445270895244cm);
\draw [line width=1.2pt] (3.02,-0.42) circle (1.4756693396557374cm);
\draw [line width=1.2pt] (8.02,-0.46) circle (0.8357032966310483cm);
\draw [line width=1.2pt] (8.02,-0.46) circle (1.360147050873545cm);
\draw [line width=1.2pt] (8.02,-0.46) circle (1.9054658223122252cm);
\draw [line width=1.2pt] (13.02,-0.56) circle (0.7727871634544661cm);
\draw [line width=1.2pt] (13.02,-0.56) circle (1.366162508635046cm);
\draw [line width=1.2pt] (13.02,-0.56) circle (1.9499743588057767cm);
\draw (-1.854718256949648,-0.161677481046377214) node[anchor=north west] {$F\; \text{and} \;G$};
\draw (-1.8866580151628973,0.9082016255720227) node[anchor=north west] {$F$};
\draw (0.06676046718120474,1.0994453930742423) node[anchor=north west] {$r_2$};
\draw (-0.43866948978894754,0.5393743596748847) node[anchor=north west] {$r_1$};
\end{tikzpicture}
\caption{First configuration}
\label{first_config}
	\end{center}
\end{figure}
%%%%%%%%%%%%%%%%%%%%%%%%%%%%%%%%%%%%%%%%%%%%%%%%%%%%%%%%%%%%%%%%%%%%%%%%
%%%%%%%%%%%%%%%%%%%%%%%%%%%%%%%%%%%%%%%%%%%%%%%%%%%%%%%%%%%%%%%%%%%%%%%%
According to \eqref{OnE_FandE_G} and \eqref{OnE_G}, we look for $F$ and $G$ under the form
\[
F(r) = \begin{cases}
C_1 + \dfrac{\nu-1}{8\nu(1-\rho)} r^2 & \text{ for }\; 0 \leq r \leq r_1, \\
C_3 - \dfrac{r^2}{8\nu} & \text{ for }\; r_1 \leq r \leq r_2, \\
0 & \text{ for }\; r \geq r_2, 
\end{cases}
\]
and
\[
G(r) = \begin{cases}
C_2 + \dfrac{\rho-\nu}{8 \nu(1-\rho)} r^2 & \text{ for }\; 0 \leq r \leq r_1, \\
0 & \text{ for } r \geq r_1.
\end{cases}
\]
In order to determine $C_1, C_2, C_3, r_1, r_2$ and especially the parameter range for which this configuration appears, we use the relations $F(0)=C_1>0$, $F(r_2)=0$, $F$ is continuous in $r=r_1$, $G(r_1)=0$ and 
$$\int_{\R^2}F\diff x= M_f, \qquad \int_{\R^2}G\diff x= M_g.$$
On the one hand, the condition $F(r_2)=0$ provides
\be 
C_3=\dfrac{1}{8\nu}r_2^2,\label{exp_C3_case1}
\ee
whereas the continuity of $F$ at $r=r_1$ yields
\be C_1+\dfrac{\nu-1}{8\nu(1-\rho)}r_1^2=C_3-\dfrac{1}{8\nu}r_1^2. \label{F_cont_r1}\ee 
To exploit the constraint on the mass for $F$, we
pass in polar coordinates and integrate with respect to $r$. This leads to
\begin{multline}\label{exp_M_f_Case1}
2\int_0^{r_1} \Big[C_1+\dfrac{\nu-1}{8\nu(1-\rho)}r^2\Big]r\diff r+
2\int_{r_1}^{r_2} \Big[C_3-\dfrac{1}{8\nu}r^2\Big]r\diff r \\
 =  C_1 r_1^2+\dfrac{\nu-1}{16\nu(1-\rho)}r_1^4+
{C_3}[r_2^2-r_1^2]-\dfrac{1}{16\nu}[r_2^4-r_1^4]=\dfrac{M_f}{\pi}.
\end{multline}
On the other hand, $G$ is decreasing with $G(r_1)=0$, hence $\rho<\nu$  and 
\be C_2=\dfrac{\nu-\rho}{8\nu(1-\rho)}r_1^2.\label{exp_C2_case1}\ee
In addition, the constraint on the mass of $G$ gives
\be 
2\int_0^{r_1} \Big[C_2+ \dfrac{\rho-\nu}{8\nu(1-\rho)}r^2\Big]r\diff r
=C_2 r_1^2+\dfrac{\rho-\nu}{16\nu(1-\rho)}r_1^4=\dfrac{M_g}{\pi}.\label{exp_M_g_Case1}
\ee 
Using \eqref{exp_C2_case1} and \eqref{exp_M_g_Case1} one gets an explicit formula for $r_1$:
\be 
r_1^4=\dfrac{16\nu (1-\rho)M_g}{\pi(\nu-\rho)}>0.
\label{exp_r1_case1}
\ee
Multiplying (\ref{F_cont_r1}) by $r_1^2$, we get 
$$C_1 r_1^2- C_3 r_1^2=\dfrac{\rho-\nu}{8\nu(1-\rho)}r_1^4,$$ 
hence, using (\ref{exp_C3_case1}), (\ref{exp_M_f_Case1}) and \eqref{exp_r1_case1}, we obtain
\be r_2^4= 
\dfrac{16\nu}{\pi}(M_f+M_g).\label{exp_r2Case1}\ee
Then
\be r_2^4 \ge r_1^4 \Longleftrightarrow \nu \ge\dfrac{\rho M_f+M_g}{M_f+M_g} = \nu_2^\star.  \label{mu_sup_valCrit_case1} \ee 
Combining (\ref{exp_C3_case1}) and (\ref{F_cont_r1}) 
we get
$$ C_1= \frac1{8\nu} \left( r_2^2 - \frac{\nu-\rho}{1-\rho} r_1^2\right),$$
and by  (\ref{exp_r2Case1}) one has 
\be F(0) = C_1> 0\Longleftrightarrow \nu < 1+(1-\rho) \dfrac{M_f}{M_g} = \nu_3^\star. \label{mu_inf_valCrit_case1}\ee 
We conclude that we have the case depicted on Figure~\ref{first_config} if and only if 
\[
\nu_2^\star \le \nu < \nu_3^\star,\quad \text{with}\quad \nu_2^\star=
\dfrac{\rho M_f+M_g}{M_f+M_g}\quad \text{and}\quad \nu_3^\star=1+(1-\rho) \dfrac{M_f}{M_g}.
\]
 Remark that if $\nu=\nu_2^\star$,  then $r_1=r_2$. 
\medskip

\paragraph{\bf Second case :} the second configuration we consider is shown on Figure~\ref{second_config} 
for some $0 < r_1 \leq r_2$. We note that  $0\in E_F\cap E_G$.
%%%%%%%%%%%%%%%%%%%%%%%%%%%%%%%%%%%%%%%%%%%%%%%%%%%%%%%%%%%%%%%%
%
%%%%%%%%%%%%%%%%%%%%%%%%%%%%%%%%%%%%%%%%%%%%%%%%%%%%%%%%%%%%%
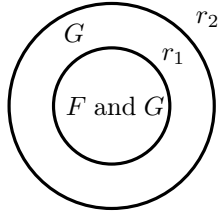
\begin{figure}[htbp] 
	\begin{center} 
\begin{tikzpicture}[scale=0.9][line cap=round,line join=round,>=triangle 45,x=1.0cm,y=1.0cm]
\clip(-2.7,-2.1) rectangle (0.6,1.2);
\draw [line width=1.2pt] (-1.04,-0.44) circle (0.8590692637965812cm);
\draw [line width=1.2pt] (-1.04,-0.44) circle (1.5153877391611692cm);
\draw [line width=1.2pt] (3.02,-0.42) circle (0.8955445270895244cm);
\draw [line width=1.2pt] (3.02,-0.42) circle (1.4756693396557374cm);
\draw [line width=1.2pt] (8.02,-0.46) circle (0.8357032966310483cm);
\draw [line width=1.2pt] (8.02,-0.46) circle (1.360147050873545cm);
\draw [line width=1.2pt] (8.02,-0.46) circle (1.9054658223122252cm);
\draw [line width=1.2pt] (13.02,-0.56) circle (0.7727871634544661cm);
\draw [line width=1.2pt] (13.02,-0.56) circle (1.366162508635046cm);
\draw [line width=1.2pt] (13.02,-0.56) circle (1.9499743588057767cm);
\draw (-1.854718256949648,-0.161677481046377214) node[anchor=north west] {$F\; \text{and} \;G$};
\draw (-1.8866580151628973,0.9082016255720227) node[anchor=north west] {$G$};
\draw (0.06676046718120474,1.0994453930742423) node[anchor=north west] {$r_2$};
\draw (-0.43866948978894754,0.5393743596748847) node[anchor=north west] {$r_1$};

\end{tikzpicture}
\caption{Second configuration}
\label{second_config}
	\end{center}
\end{figure}
%%%%%%%%%%%%%%%%%%%%%%%%%%%%%%%%%%%%%%%%%%%%%%%%%%%%%%%%%%%%%%%%%%%
%%%%%%%%%%%%%%%%%%%%%%%%%%%%%%%%%%%%%%%%%%%%%%%%%%%%%%%%%%%%%%%%%%%%%
According to \eqref{OnE_FandE_G} and \eqref{OnE_F},  we look for $F$ and $G$ under the form
\[
F(r) = 
\begin{cases}
C_1 + \dfrac{\nu-1}{8\nu(1-\rho)} r^2 & \text{ for }\; 0 \leq r \leq r_1, \\
0 & \text{ for }\; r \geq r_1,
\end{cases}
\] 
and 
\[
G(r) = 
\begin{cases}
C_2 + \dfrac{\rho-\nu}{8\nu(1-\rho)} r^2 & \text{ for }\; 0 \leq r \leq r_1, \\
C_4 - \dfrac1{8} r^2 & \text{ for }\; r_1 \leq r \leq r_2, \\
0 & \text{ for }\; r \geq r_2.
\end{cases}
\] 
This case is very similar to the first case, the roles of $F$ and $G$ being exchanged. Since $F$ is decreasing, there holds $\nu < 1$. Reproducing the calculations of the previous case provides that
$$ 
r_1^4=\dfrac{16\nu (1-\rho)M_f}{\pi(1-\nu)}>0,
\qquad 
r_2^4 = \dfrac{16}{\pi} (\rho M_f + M_g), 
$$ 
\[
C_1 = \dfrac{1-\nu}{8\nu(1-\rho)} r_1^2, 
\qquad
C_4 = \frac{1}{8} r_2^2,
\qquad
C_2 = C_4 - \rho C_1.
\]
The conditions $G(0)=C_2 > 0$ and $r_1 \le r_2$ show 
that the  case of Figure~\ref{second_config} occurs if and only if
$$ 
\nu_1^\star < \nu \le \nu_2^\star,\quad \text{with}\quad {\nu_1^\star=\dfrac{\rho^2 M_f}{\rho M_f+(1-\rho) M_g}} \quad \text{and}\quad \nu_2^\star=\dfrac{\rho M_f+M_g}{M_f+M_g}.
$$

\medskip

\paragraph{\bf Third case:} we consider the  configuration shown on  Figure~\ref{third_config}. We note that $0\in E_F\cap E_G^c$.
%%%%%%%%%%%%%%%%%%%%%%%%%%%%%%%%%%%%%%%%%%%%%%%%%%%%%%%%%%%%%%%%
%
%%%%%%%%%%%%%%%%%%%%%%%%%%%%%%%%%%%%%%%%%%%%%%%%%%%%%%%%%%%%%%%%
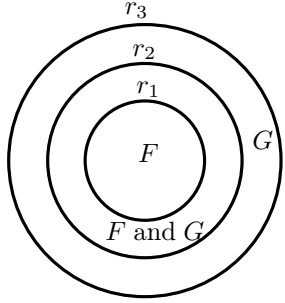
\begin{figure}[htb] 
	\begin{center} 

\begin{tikzpicture}[scale=0.95][line cap=round,line join=round,>=triangle 45,x=1.0cm,y=1.0cm]
\clip(6.,-2.5) rectangle (10.1,1.8);
\draw [line width=1.2pt] (-1.04,-0.44) circle (0.8590692637965812cm);
\draw [line width=1.2pt] (-1.04,-0.44) circle (1.5153877391611692cm);
\draw [line width=1.2pt] (3.02,-0.42) circle (0.8955445270895244cm);
\draw [line width=1.2pt] (3.02,-0.42) circle (1.4756693396557374cm);
\draw [line width=1.2pt] (8.02,-0.46) circle (0.8357032966310483cm);
\draw [line width=1.2pt] (8.02,-0.46) circle (1.360147050873545cm);
\draw [line width=1.2pt] (8.02,-0.46) circle (1.9054658223122252cm);
\draw [line width=1.2pt] (13.02,-0.56) circle (0.7727871634544661cm);
\draw [line width=1.2pt] (13.02,-0.56) circle (1.366162508635046cm);
\draw [line width=1.2pt] (13.02,-0.56) circle (1.9499743588057767cm);
\draw (7.794094665664915,-0.10265828836828104) node[anchor=north west] {$F$};
\draw (7.338927668875078,-1.180838740523188) node[anchor=north west] {$F\; \text{and} \;G$};
\draw (9.39672426746807,0.08473328324567944) node[anchor=north west] {$G$};
\draw (7.771152243699202,0.770538214602822) node[anchor=north west] {$r_1$};
\draw (7.716511167269997,1.3021248548596398) node[anchor=north west] {$r_2$};
\draw (7.607229014411585,1.877554811829792) node[anchor=north west] {$r_3$};
\end{tikzpicture}
		\caption{Third configuration}
		\label{third_config}
	\end{center}
\end{figure}
%%%%%%%%%%%%%%%%%%%%%%%%%%%%%%%%%%%%%%%%%%%%%%%%%%%%%%%%%%%%%%%%%%%%%%%%%
%%%%%%%%%%%%%%%%%%%%%%%%%%%%%%%%%%%%%%%%%%%%%%%%%%%%%%%%%%%%%%%%%%%%%%%
According to \eqref{OnE_FandE_G}--\eqref{OnE_G}, $F$ and $G$ are given by
\[
F(r) = \begin{cases}
C_3 - \dfrac{r^2}{8\nu} & \text{ for }\; 0 \leq r \leq r_1, \\
C_1 + \dfrac{\nu-1}{8\nu(1-\rho)} r^2 & \text{ for }\; r_1 \leq r \leq r_2, \\
0 & \text{ for }\; r \geq r_2, 
\end{cases}
\]
and
\[
G(r) = \begin{cases}
0 & \text{ for }\; 0 \leq r \leq r_1, \\
C_2 + \dfrac{\rho-\nu}{8 \nu(1-\rho)} r^2 & \text{ for }\; r_1 \leq r \leq r_2, \\
C_4 - \dfrac{1}{8} r^2 & \text{ for } r_2 \leq r \leq r_3.
\end{cases}
\]
Observe first that $G$ shall increase on $(r_1,r_2)$ which implies that $\nu\le\rho<1$. Owing to the continuity of $F$ and $G$ and their vanishing properties, we require that 
\begin{align}
& C_1 = \dfrac{1-\nu}{8\nu(1-\rho)} r_2^2, \qquad C_2 = - \dfrac{\rho-\nu}{8\nu(1-\rho)} r_1^2, \qquad C_4 = \dfrac{r_3^2}{8}, \label{Z31} \\
& C_2 - C_4 = - \dfrac{\rho(1-\nu)}{8\nu(1-\rho)} r_2^2, \qquad C_1-C_3 = \dfrac{\rho-\nu}{8\nu(1-\rho)} r_1^2. \label{Z32}
\end{align}
Combining \eqref{Z31} and \eqref{Z32} leads us to 
%$$
%- \dfrac{\rho(1-\nu)}{8\nu(1-\rho)} r_2^2 = C_2-C_4 = - \dfrac{\rho-\nu}{8\nu(1-\rho)} r_1^2 - \dfrac{r_3^2}{8},
%$$
%hence
\begin{equation}
%(\rho-\nu) r_1^2- \rho(1-\nu) r_2^2 + \nu(1-\rho) r_3^2 = 0. 
r_3^2 = \frac{\rho(1-\nu)}{\nu(1-\rho)}r_2^2 - \frac{\rho - \nu}{\nu(1-\rho)}r_1^2.
\label{Z33}
\end{equation}
In addition, the mass constraints give
\begin{align}
16\nu(1-\rho) \dfrac{M_f}{\pi} & = - (\rho-\nu) r_1^4 + (1-\nu) r_2^4, \label{Z34} \\
16\nu(1-\rho) \dfrac{M_g}{\pi} & = (\rho-\nu) r_1^4 - \rho(1-\nu) r_2^4 + \nu(1-\rho) r_3^4. \label{Z35}
\end{align}
We use \eqref{Z33} to substitute $r_3^4$ in \eqref{Z35}, providing the relation
\[
16\nu^2(1-\rho)^2 \dfrac{M_g}{\pi} = \rho(\rho-\nu)(1-\nu)\left(r_2^2 - r_1^2\right)^2.
\]
Since $r_2 \geq r_1$, one gets $\nu<\rho$ and
\be\label{eq:Z36}
r_2^2 = r_1^2 + 4 \nu(1-\rho)\sqrt{\frac{M_g}{\pi(\rho-\nu)\rho(1-\nu)}}.
\ee
Combining the relations~\eqref{eq:Z36} and~\eqref{Z34}, we obtain that $r_1^2$ is the root of the 
polynomial of degree two, i.e., 
\be\label{eq:Z37}
r_1^4 - S r_1^2 + P=0,
\ee
with
\[
S= - 8 \nu \sqrt{\dfrac{M_g(1-\nu)}{\pi(\rho - \nu) \rho}}, \qquad 
P= \frac{16 \nu}\pi \left(\frac{\nu(1-\rho)M_g}{\rho(\rho-\nu)}-M_f\right).
\]
Since $S<0$, the polynomial $X^2 - SX + P$ admits one nonnegative root if and only if $P \leq 0$, i.e., 
if and only if 
\[
\nu \leq \nu_1^\star = \frac{\rho^2 M_f}{\rho M_f + (1-\rho) M_g}.
\]
Therefore, \eqref{eq:Z37} has no real solution if $\nu > \nu_1^\star$, whereas it admits one unique nonnegative solution if $\nu\leq \nu_1^\star$, given by 
\[
r_1 = 2 \sqrt{\frac{\nu}{\sqrt\pi} \left(\sqrt{\frac{M_g}{\rho} + \frac{M_f}{\nu}}-\sqrt{\frac{(1-\nu)M_g}{\rho(\rho-\nu)}}\right)}.
\]

\medskip

\paragraph{\bf Fourth case :} we consider the last configuration shown on Figure~\ref{fourth_config} 
for some $0 \leq r_1 < r_2 < r_3$. We note that $0\in E_G\cap E_F^c$. 
%%%%%%%%%%%%%%%%%%%%%%%%%%%%%%%%%%%%%%%%%%%%%%%%%%%%%%%%%%%%%%%%%%%%
%
%%%%%%%%%%%%%%%%%%%%%%%%%%%%%%%%%%%%%%%%%%%%%%%%%%%%%%%%%%%%%%%%%%%%%%
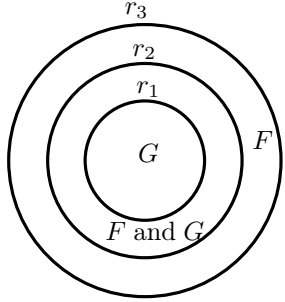
\begin{figure}[htb] 
	\begin{center} 
	\begin{tikzpicture}[scale=0.95][line cap=round,line join=round,>=triangle 45,x=1.0cm,y=1.0cm]
\clip(6.,-2.5) rectangle (10.1,1.8);
\draw [line width=1.2pt] (-1.04,-0.44) circle (0.8590692637965812cm);
\draw [line width=1.2pt] (-1.04,-0.44) circle (1.5153877391611692cm);
\draw [line width=1.2pt] (3.02,-0.42) circle (0.8955445270895244cm);
\draw [line width=1.2pt] (3.02,-0.42) circle (1.4756693396557374cm);
\draw [line width=1.2pt] (8.02,-0.46) circle (0.8357032966310483cm);
\draw [line width=1.2pt] (8.02,-0.46) circle (1.360147050873545cm);
\draw [line width=1.2pt] (8.02,-0.46) circle (1.9054658223122252cm);
\draw [line width=1.2pt] (13.02,-0.56) circle (0.7727871634544661cm);
\draw [line width=1.2pt] (13.02,-0.56) circle (1.366162508635046cm);
\draw [line width=1.2pt] (13.02,-0.56) circle (1.9499743588057767cm);

\draw (7.794094665664915,-0.10265828836828104) node[anchor=north west] {$G$};
\draw (7.338927668875078,-1.180838740523188) node[anchor=north west] {$F\; \text{and} \;G$};
\draw (9.39672426746807,0.08473328324567944) node[anchor=north west] {$F$};
\draw (7.771152243699202,0.770538214602822) node[anchor=north west] {$r_1$};
\draw (7.716511167269997,1.3021248548596398) node[anchor=north west] {$r_2$};
\draw (7.607229014411585,1.877554811829792) node[anchor=north west] {$r_3$};

\end{tikzpicture}
\caption{Fourth configuration }
\label{fourth_config}
	\end{center}
\end{figure}
%%%%%%%%%%%%%%%%%%%%%%%%%%%%%%%%%%%%%%%%%%%%%%%%%%%%%%%%%%%%%%%%%%%%%%%%%%
%%%%%%%%%%%%%%%%%%%%%%%%%%%%%%%%%%%%%%%%%%%%%%%%%%%%%%%%%%%%%%%%%%%%%%%%%%

According to \eqref{OnE_FandE_G}--\eqref{OnE_G}, $F$ and $G$ are given by
\[
F(r) = \begin{cases}
0 & \text{ for }\; 0 \leq r \leq r_1, \\
C_1 + \dfrac{\nu-1}{8\nu(1-\rho)} r^2 & \text{ for }\; r_1 \leq r \leq r_2, \\
C_3 - \dfrac{r^2}{8\nu} & \text{ for }\; r_2 \leq r \leq r_3, 
\end{cases}
\]
and
\[
G(r) = \begin{cases}
C_4 - \dfrac{1}{8} r 2 & \text{ for } 0 \leq r \leq r_1, \\
C_2 + \dfrac{\rho-\nu}{8 \nu(1-\rho)} r^2 & \text{ for }\; r_1 \leq r \leq r_2, \\
0 & \text{ for }\; r \geq r_2.
\end{cases}
\]
First of all, we note that $F$ shall increase on $(r_1,r_2)$ which implies that $\nu\ge 1$. Owing to the continuity of $F$ and $G$ and their vanishing properties, we require that 
\begin{align}
& C_1 = - \dfrac{\nu-1}{8\nu(1-\rho)} r_1^2, \qquad C_2 = \dfrac{\nu-\rho}{8\nu(1-\rho)} r_2^2, \qquad C_3 = \dfrac{r_3^2}{8\nu}, \label{Z41} \\
& C_3-C_1 = \dfrac{\nu-\rho}{8\nu(1-\rho)} r_2^2, \qquad 
C_2 - C_4 = - \dfrac{\rho(\nu-1)}{8\nu(1-\rho)} r_1^2. \label{Z42}
\end{align}
It follows from \eqref{Z41} and \eqref{Z42} that
$$
\dfrac{\nu-\rho}{8\nu(1-\rho)} r_2^2 = C_3 - C_1 = \dfrac{r_3^2}{8\nu} + \dfrac{\nu-1}{8\nu(1-\rho)} r_1^2,
$$
hence
\begin{equation}
(\nu-1) r_1^2 - (\nu-\rho) r_2^2 + (1-\rho) r_3^2 = 0. \label{Z43}
\end{equation}
Moreover, the mass constraints give
\begin{equation}
(\nu-1) r_1^4 - (\nu-\rho) r_2^4 + (1-\rho) r_3^4 = 16\nu(1-\rho) \dfrac{M_f}{\pi} \label{Z44}
\end{equation}
and
\begin{equation}
- \rho(\nu-1) r_1^4 + (\nu-\rho) r_2^4 = 16\nu(1-\rho) \dfrac{M_g}{\pi}. \label{Z45}
\end{equation}
It readily follows from \eqref{Z44} and \eqref{Z45} that
\begin{align}
r_2^4 & = \dfrac{16\nu(1-\rho)}{\nu-\rho} \dfrac{M_g}{\pi} + \dfrac{\rho(\nu-1)}{\nu-\rho} r_1^4, \label{Z46} \\
r_3^4 & = 16\nu \dfrac{M_f+M_g}{\pi} - (\nu-1) r_1^4. \label{Z47}
\end{align}
It results from (\ref{Z43}) that
\[ r_3^4=\left(\dfrac{1-\nu}{1-\rho}\right)^2r_1^4+\left(\dfrac{\nu-\rho}{1-\rho}\right)^2r_2^4-2\dfrac{(\nu-1)(\nu-\rho)}{(1-\rho)^2}r_1^2r_2^2,\]
hence
\be \dfrac{(\nu-1)(\nu-\rho)}{(1-\rho)^2}(r_2^2 -r_1^2)^2=r_3^4+\dfrac{\nu-1}{1-\rho}r_1^4-\dfrac{\nu-\rho}{1-\rho}r_2^4.\label{Z60}\ee
Multiplying (\ref{Z60}) by $(1-\rho)$ and using (\ref{Z44}), one has
\be (r_2^2-r_1^2)^2=\dfrac{16\nu (1-\rho)^2M_f}{\pi (\nu-1)(\nu-\rho)}.\label{Z61}\ee 
Since $r_2 \geq r_1$, one gets 
\be\label{Z62}
r_2^2 = r_1^2 + 4 (1-\rho)\sqrt{\frac{\nu M_f}{\pi(\nu-\rho)(\nu-1)}}.
\ee
Combining the relations~\eqref{Z62} and~\eqref{Z45}, we obtain that $r_1^2$ is the root of the 
polynomial of degree two, i.e., 
\be\label{Z63}
r_1^4 - S r_1^2 + P=0,
\ee
with
\[
S= - 8  \sqrt{\dfrac{M_f(\nu-\rho)}{\pi(\nu -1) \nu}}, \qquad 
P= \frac{16}{\pi} \left(\frac{(1-\rho)M_f}{\nu-1}-M_g\right).
\]
Since $S<0$, the polynomial $X^2 - SX + P$ admits one nonnegative root if and only if $P \leq 0$, i.e., 
if and only if 
\[
\nu \geq \nu_3^\star =1+ (1-\rho)\frac{M_f}{M_g}.
\]
Therefore, \eqref{Z63} has no real solution if $\nu < \nu_3^\star$, whereas it admits one unique nonnegative solution if $\nu\geq \nu_3^\star$, given by 
\[
r_1 = 2 \sqrt{\frac{1}{\sqrt\pi} \left(\sqrt{M_g + M_f\frac{\rho}{\nu}}-\sqrt{\frac{(\nu-\rho)M_f}{\nu(\nu-1)}}\right)}.
\]

\begin{rmk}\label{rmk:uniqueness}
As a consequence of the case study carried out above, there exists a unique solution $(F,G)\in \K_{M_f} \times \K_{M_g}$ to the problem~\eqref{system_stat}.  Owing to Proposition~\ref{prop:minimizer}, the unique minimizer of the energy $\Ee$ in $\K_{M_f} \times \K_{M_g}$ satisfies~\eqref{system_stat}. Consequently, being a minimizer of the minimization problem~\eqref{pb_min} is equivalent to solving~\eqref{system_stat}. 
This concludes the proofs of Theorem~\ref{thm:exisUniq_profil} and Theorem~\ref{thm:conv}. 
\end{rmk}

% ======================================
% 			Numerical scheme
% ======================================
\section{Numerical investigation}\label{numer_invest}

In this section, we present numerical simulations of the system~\eqref{system2}. 
We are interested in its long-time behavior and thus compute the numerical solution until stabilization.  
Since our study is widely based on the use of energy and dissipation estimates, we make use of the 
upstream mobility finite volume scheme studied in~\cite{moifvca2, intrusion} and described in the next section.

\subsection{The numerical scheme}\label{def_scheme}
We detail here the discretization of the problem (\ref{system2}) we use for the numerical simulations. The time discretization relies on backward Euler scheme, while the space discretization relies on a finite volume approach (see, e.g., \cite{Eymard3}), with a two-point flux approximation and an upstream choice for the mobility.

Recall that the minimizers have compact support by Theorem~\ref{thm:exisUniq_profil}, which allows us to perform the simulations on an open bounded polygonal domain $\O\subset \R^2$ which is chosen to be larger than the support of both the initial data and the final states. It is actually always possible to take here $\O=(0,1)^2$ at the expense of reducing the masses $M_f$ and $M_g$ (only the ratio $M_f/M_g$ 
has an influence on the shape of the minimizers). Practically, no-flux boundary conditions across $\p\O$ are prescribed in the numerical method. 

An admissible mesh of $\O$ is given by a family $\T$ of control volumes (open and convex polygons), a family $\E$ of edges and a family of points $(x_K)_{K\in \T}$ which satisfy Definition 9.1 in \cite{Eymard3}. This definition implies that the straight line between two neighboring centers of cells $(x_K,x_L)$ is orthogonal to the edge $\sigma =K|L$ separating the cell $K$ and the cell $L$.

We denote the set of interior edges by $\E_{\text{int}}$. For a control volume $K\in\T$, we denote the set of its edges by $\E_K$ and the set of its interior edges by $\E_{K,\text{int}}$. For $\sigma \in \E_{\text{int},K}$ with $\sigma=K|L$, we define $d_{\sigma}=d(x_K,x_L)$ and the transmissibility coefficient 
\[
\tau_\sigma=\dfrac{m(\sigma)}{d_\sigma}, 
\]
where $d$ denotes the distance in $\R^2$ and $m$ the Lebesgue measure in $\R^2$ or $\R$. 

In the simulation, we use variable time steps $\Delta t_n = t^{n+1}-t^n$ with $t^0 = 0$.

The quantity $f_K^n$, $g_K^n$ and $b_K$ approximate the value of $f$, $g$ and $b$, respectively, in the circumcenter of $K$ at time $t^n$. It is given as a data for $n=0$ and, for $n\geq 1$, as a solution to the nonlinear system 
\begin{multline}
m(K)\dfrac{f_K^{n+1}-f_K^n}{\Delta t} \\ +\sum_{\sigma\in\E_{\text{int},K}}\tau_\sigma f_\sigma^{n+1}\nu\Bigg((f_K^{n+1}-f_L^{n+1})+(g_K^{n+1}-g_L^{n+1})+\dfrac{1}{\nu}(b_K-b_L)\Bigg)=0, 
\label{f_Kscheme}
\end{multline}
and
\begin{multline} 
m(K)\dfrac{g_K^{n+1}-g_K^n}{\Delta t} \\ +\sum_{\sigma\in\E_{\text{int},K}}\tau_\sigma g_\sigma^{n+1}\Bigg(\rho(f_K^{n+1}-f_L^{n+1})+(g_K^{n+1}-g_L^{n+1})+(b_K-b_L)\Bigg)=0,
\label{g_Kscheme}
\end{multline}
for $K\in\T$,  with an upstream choice for the mobilities 
\begin{equation}
f_{\sigma}^{n+1}=
\begin{cases}
(f_K^{n+1})^+ &\text{if}\;\; (f_K^{n+1}-f_L^{n+1})+(g_K^{n+1}-g_L^{n+1})+\dfrac{1}{\nu}(b_K-b_L)\geq 0 , \\
(f_L^{n+1})^+ &\text{if}\;\; (f_K^{n+1}-f_L^{n+1})+(g_K^{n+1}-g_L^{n+1})+\dfrac{1}{\nu}(b_K-b_L)< 0.
\end{cases}
\label{f_sigma}
\end{equation} 
and
\begin{equation}
g_{\sigma}^{n+1}=
\begin{cases}
(g_K^{n+1})^+ &\;\text{if}\;\;\rho (f_K^{n+1}-f_L^{n+1})+(g_K^{n+1}-g_L^{n+1})+(b_K-b_L)\geq 0 , \\
(g_L^{n+1})^+ &\;\text{if}\;\;\rho(f_K^{n+1}-f_L^{n+1})+(g_K^{n+1}-g_L^{n+1})+(b_K-b_L)< 0,
\end{cases}
\label{g_sigma}
\end{equation}
where $x^+=\max(0,x)$ and $\sigma = K|L$.
The discretization of the steady state problem \eqref{system_stat} is given
by the following set of nonlinear equations:
\begin{align}
\sum_{\sigma\in\E_{\text{int},K}}\tau_\sigma F_\sigma\nu\Bigg((F_K-F_L)+(G_K-G_L)+\dfrac{1}{\nu}(b_K-b_L)\Bigg)=0, 
\label{F_Kscheme}
\end{align}
and
\begin{align} 
\sum_{\sigma\in\E_{\text{int},K}}\tau_\sigma G_\sigma\Bigg(\rho(F_K-F_L)+(G_K-G_L)+(b_K-b_L)\Bigg)=0,
\label{G_Kscheme}
\end{align}
where $F_\sigma$ and $G_\sigma$ are defined  in a similar way as $f_\sigma^{n+1}$ and $g_\sigma^{n+1}$ above. 
The system~\eqref{F_Kscheme}--\eqref{G_Kscheme} is underdetermined and one has to add the constraints
\be\label{eq:mass-FG}
F_K \geq 0, \quad G_K \geq 0, \quad \sum_{K\in\T}m(K) F_K = M_f, \quad \sum_{K\in\T}m(K) G_K = M_g.
\ee

The relative energy $\Ee(f,g|F,G)$ of a solution $(f,g)$ to \eqref{system2} with respect to the stationary state $(F,G)$ is defined by
$$
\Ee(f,g|F, G)=\int_{\R^2} \Big( E(f,g)-E(F, G)\Big)\diff x=\Ee(f,g)-\Ee(F,G).
$$
It is a classical tool to study the large time behavior of problems with a gradient flow structure since 
the relative energy is decaying along time:
\[
\Ee(f(t),g(t)|F, G) + \int_0^t \mathcal{I}(f,g)(\tau)\diff \tau \leq \Ee(f_0,g_0|F, G), \qquad t \geq 0,
\]
by \eqref{eq:NRJ}. Since $(F,G)$ is a minimizer of $\Ee$ in $\K_{M_f} \times \K_{M_g}$ by Proposition~\ref{prop:minimizer} and Remark~\ref{rmk:uniqueness}, then
$$ \Ee(f,g|F, G)=\Ee(f,g)-\Ee(F,G) \geq 0.$$
Owing to \eqref{eq:NRJ} and Theorem~\ref{thm:conv}, the relative energy shall decay to zero as time goes to infinity. We investigate in Section~\ref{ssec:simu} at what speed this convergence occurs. Note that in our case, the relative energy reduces to 
$$
\Ee(f,g|F, G)= \int_{\R^2}\left( \frac{\rho}2 ( f-F + g-G)^2 + \frac{1-\rho}2 (g-G)^2 \right) \diff x
$$
which, according to~\eqref{controlnormeL2}, is equivalent to the square of the $L^2(\R^2)^2$ distance 
between $(f,g)$ and $(F,G)$.

For the computations, we introduce a discrete version of the relative energy functional :
\begin{align*}\Ee^n = &
\sum_{K\in \T} m(K)\Bigg[\dfrac{\rho}{2}\big(f_K^n+g_K^n)^2+\dfrac{1-\rho}{2}(g_K^n)^2+b_K\Big(\dfrac{\rho}{\nu}f_K^n+g_K^n\Big)\Bigg] \\ 
& - 
\sum_{K\in \T} m(K)\Bigg[\dfrac{\rho}{2}\big(F_K+G_K)^2+\dfrac{1-\rho}{2}(G_K)^2+b_K\Big(\dfrac{\rho}{\nu}F_K+G_K\Big)\Bigg], 
\qquad n \geq 0.
\end{align*}
Since the scheme is energy diminishing \cite{intrusion}, this quantity decreases when $n$ increases. 
Moreover, one can transpose to the discrete setting the proof of Proposition~\ref{prop:minimizer} and establish that 
the unique minimizer $\left({(F_K)}_K,{(G_K)}_K\right)$ of the energy is a solution to the 
scheme~\eqref{F_Kscheme}--\eqref{eq:mass-FG}.
If  $\left({(F_K)}_K,{(G_K)}_K\right)$ used in the definition of $\Ee^n$ is this minimizer (we have not proven that 
the steady discrete problem \eqref{F_Kscheme}--\eqref{eq:mass-FG} admits a unique solution), 
then $\Ee^n$ remains positive and converges to zero as $n\to\infty$. This property is observed in the numerical simulations of the next section.  

\subsection{Numerical simulations}\label{ssec:simu}
Our scheme leads to a nonlinear system that we solve thanks to the Newton-Raphson method. In  our test case, the domain is the unit square, i.e., $\O=(0,1)^2$.  We consider an admissible triangular mesh made of $14336$ triangles. 
We use a mesh coming from the 2D benchmark on anisotropic diffusion problems \cite{Herbin}.
For the evolutive solutions, 
we use an adaptive time step procedure in the practical implementation  in order to increase the robustness of 
the algorithm and to ensure the convergence of the Newton-Raphson iterative procedure. More precisely, 
we associate a maximal time step $\Delta t_{\max}=2.10^{-4}$ for the mesh. If the Newton-Raphson method fails to 
converge after 30 iterations ---we choose that the $\ell^\infty$ norm of the residual has to be smaller than $10^{-9}$ as stopping criterion---, the time step is divided by two. If the Newton-Raphson method 
converges, the first time step is multiplied by two and projected on $[0,\Delta t_{\max}]$.
The first time step $\Delta t$ is equal to $\Delta t_{\max}$ in  the test case presented below.

We perform the numerical experiments with the following data
$$ b(x,y)=\dfrac{1}{8}\Big((x-1/2)^2+(y-1/2)^2 \Big),\quad \rho=0.9,$$
and as an initial condition we take
\begin{align*}
f_0(x,y)=&\frac13\Big(\dfrac{1}{16}-(x-2/7)^2-(y-2/7)^2\Big)^+,\\
g_0(x,y)=&\frac13\Big(\dfrac{1}{16}-(x-5/7)^2-(y-5/7)^2\Big)^+.
\end{align*}
In this case we have $M_f=M_g$ then
$$\nu_1^\star=0.81,\quad \nu_2^\star=0.95\quad \text{and}\;\;\nu_3^\star=1.1 .$$
Note that $f_0$ and $g_0$ are not radially symmetric with respect to $(1/2, 1/2)$.

We represent in Figure~\ref{fig:cas_mu=0.50} to \ref{fig:cas_mu=2} the self-similar profiles. Following the values of $\nu$  these figures  confirm the discussion above on the shape of the steady states.

%%%%%%%%%%%%%%%%%%%%%%%%%%%%%%%%%%%%%%%%%%%%%%%%%%%%%%%%%%%%%%
%%%%%%%%%%%%%%%%%%%%%%%%%%%%%%%%%%%%%%%%%%%%%%%%%%%%%%%
\begin{figure}[!htbp]
\begin{center}
\begin{subfigure}{0.45\textwidth}
\includegraphics[width=0.9\linewidth]{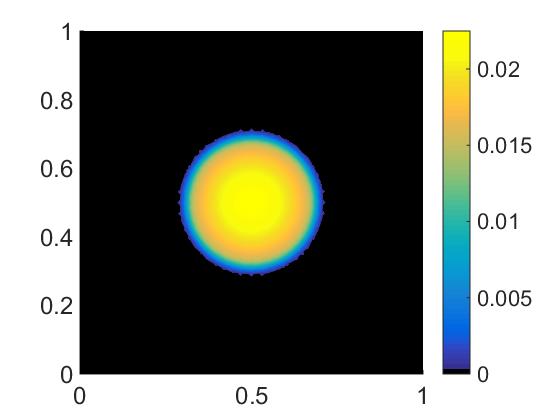} 
\caption{Profile of $F$}

\end{subfigure}
\begin{subfigure}{0.45\textwidth}
\includegraphics[width=0.9\linewidth]{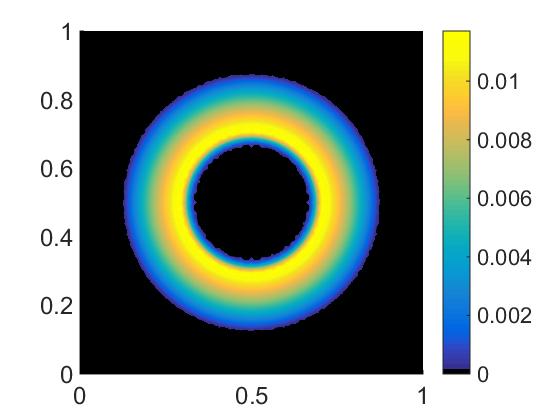}
\caption{Profile of $G$}

\end{subfigure}
 \begin{subfigure}{0.45\textwidth}
\includegraphics[width=0.9\linewidth]{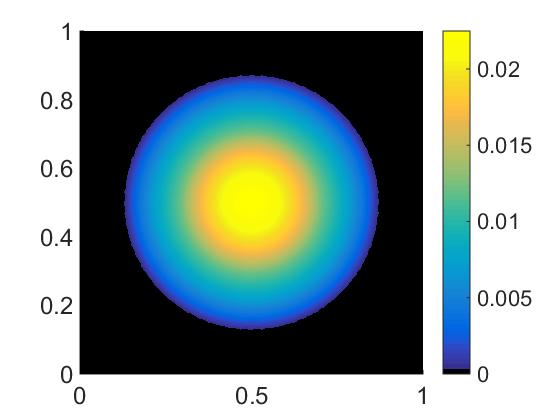} 
\caption{Profile of $F+G$}

\end{subfigure}
\begin{subfigure}{0.45\textwidth}
\centering
\begin{tikzpicture}[scale=0.8][line cap=round,line join=round,>=triangle 45,x=1.0cm,y=1.0cm]
\clip(6.,-2.5) rectangle (10.1,1.8);
\draw [line width=1.2pt] (-1.04,-0.44) circle (0.8590692637965812cm);
\draw [line width=1.2pt] (-1.04,-0.44) circle (1.5153877391611692cm);
\draw [line width=1.2pt] (3.02,-0.42) circle (0.8955445270895244cm);
\draw [line width=1.2pt] (3.02,-0.42) circle (1.4756693396557374cm);
\draw [line width=1.2pt] (8.02,-0.46) circle (0.8357032966310483cm);
\draw [line width=1.2pt] (8.02,-0.46) circle (1.360147050873545cm);
\draw [line width=1.2pt] (8.02,-0.46) circle (1.9054658223122252cm);
\draw [line width=1.2pt] (13.02,-0.56) circle (0.7727871634544661cm);
\draw [line width=1.2pt] (13.02,-0.56) circle (1.366162508635046cm);
\draw [line width=1.2pt] (13.02,-0.56) circle (1.9499743588057767cm);
\draw (7.794094665664915,-0.10265828836828104) node[anchor=north west] {$F$};
\draw (7.538927668875078,-1.280838740523188) node[anchor=north west] {$F,G$};
\draw (7.79672426746807,1.45473328324567944) node[anchor=north west] {$G$};
\end{tikzpicture}
\caption{The first configuration}
\end{subfigure}
\caption{Self-similar profiles for $\nu=0.50$}
\label{fig:cas_mu=0.50}
\end{center}
\end{figure}
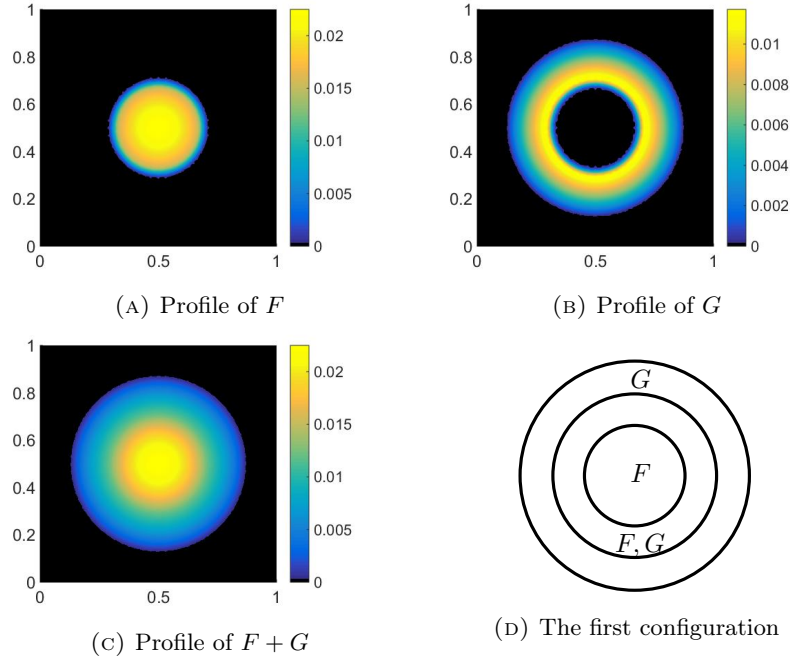
%%%%%%%%%%%%%%%%%%%%%%%%%%%%%%%%%%%%%%%%%%%%%
%%%%%%%%%%%%%%%%%%%%%%%%%%%%%%%%%%%%%%%%%%%%
\begin{figure}[!htbp]
\centering 
\begin{subfigure}{0.55\textwidth}
\begin{tikzpicture}[scale=1.5][line cap=round,line join=round,>=triangle 45,x=1.0cm,y=1.0cm]
\clip(-0.2,-0.9) rectangle (5.,2.5);
\draw [line width=1.2pt] (2.1915382332828153,0.38)-- (2.1915382332828153,-0.46);
\draw [line width=1.2pt] (4.5920622420021475,0.3896312668968353)-- (4.5920622420021475,-0.41036873310316474);
\draw (0.9650587165042536,1.6793763101280743) node[anchor=north west] {$F$};
\draw (0.8780324655952942,1.0282100543713307) node[anchor=north west] {$F,G$};
\draw (0.9650587165042536,2.4109604854803544) node[anchor=north west] {$G$};
\draw (3.2374108235715244,1.6600371432594168) node[anchor=north west] {$F, G$};
\draw (5.712824182759699,1.650367559825088) node[anchor=north west] {$F,G$};
\draw (8.30427254315982,1.0992013040683442) node[anchor=north west] {$F,G$};
\draw (8.39129879406878,2.081951735177368) node[anchor=north west] {$F$};
\draw (3.247080407005853,2.323934234571395) node[anchor=north west] {$G$};
\draw (4.765205006195476,0.40299129679666806) node[anchor=north west] {$\nu$};
\draw (7.182600864777679,-0.6316541306765173) node[anchor=north west] {1.1};
\draw (4.504126253468598,-0.650993297545175) node[anchor=north west] {0.95};
\draw (2.0287128942804227,-0.6703324644138327) node[anchor=north west] {0.81};
\draw (-0.09859546127191571,-0.65255787734366575) node[anchor=north west] {$0$};
\draw [line width=1.2pt] (0.0029784765304113965,0.4967681822044667)-- (0.0029784765304113965,-0.4083766692161724);
\draw [->,line width=1.2pt] (0.0029784765304113965,0.) -- (10.,0.);
\draw [line width=1.2pt] (1.1280257214749543,1.3995659603922253) circle (1.cm);
\draw [line width=1.2pt] (1.1280257214749543,1.3995659603922253) circle (0.7cm);
\draw [line width=1.2pt] (1.1280257214749543,1.3995659603922253) circle (0.4cm);
\draw [line width=1.2pt] (3.484473766658342,1.4021217888328459) circle (1.cm);
\draw [line width=1.2pt] (3.484473766658342,1.4021217888328459) circle (0.5cm);
\draw [line width=1.2pt] (6.240414280138422,1.3598669769890335) circle (0.5cm);
\draw [line width=1.2pt] (6.240414280138422,1.3598669769890335) circle (1.cm);
\draw [line width=1.2pt] (8.559347958987894,1.3641887932766428) circle (0.4cm);
\draw [line width=1.2pt] (8.559347958987894,1.3641887932766428) circle (0.7cm);
\draw [line width=1.2pt] (8.559347958987894,1.3641887932766428) circle (1.cm);
\draw [line width=1.2pt] (7.279860913299552,0.4058946826221381)-- (7.279860913299552,-0.3771857952123998);
\draw [->,line width=1.2pt] (0.0029784765304113965,-5.228373137344811E-4) -- (5.006644834967234,0.);
\end{tikzpicture}

\end{subfigure}

\vspace{0.9cm}

\begin{tabular}{ccc}

\includegraphics[width=0.3\linewidth]{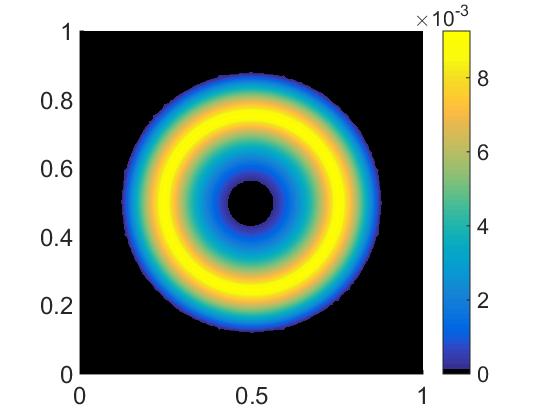} &

\includegraphics[width=0.3\linewidth]{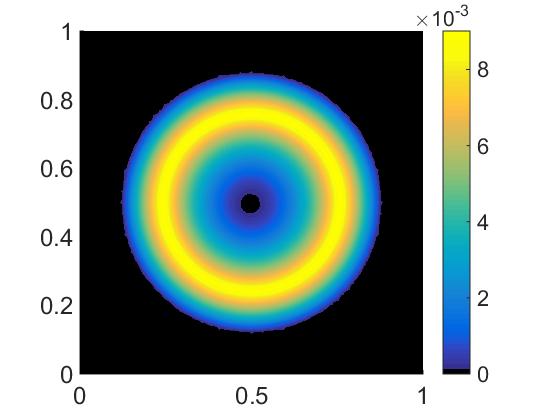}&

\includegraphics[width=0.3\linewidth]{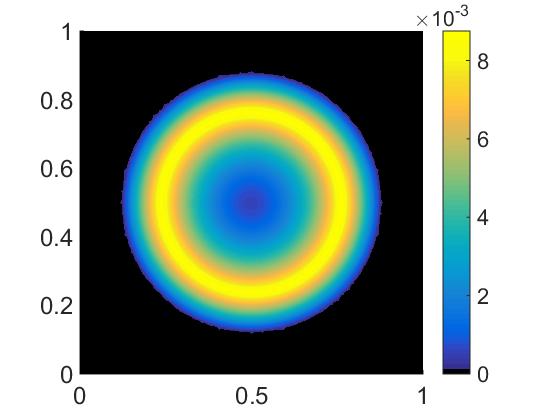} \\
	{Profile of $G$, $\nu=0.80$}&{Profile of $G$,  $\nu=\nu_1^\star=0.81$}&{Profile of $G$, $\nu=0.82$}\\

\end{tabular}
\caption{Topological change of  $E_G$ near the critical value $\nu_1^\star=0.81$}

\end{figure}
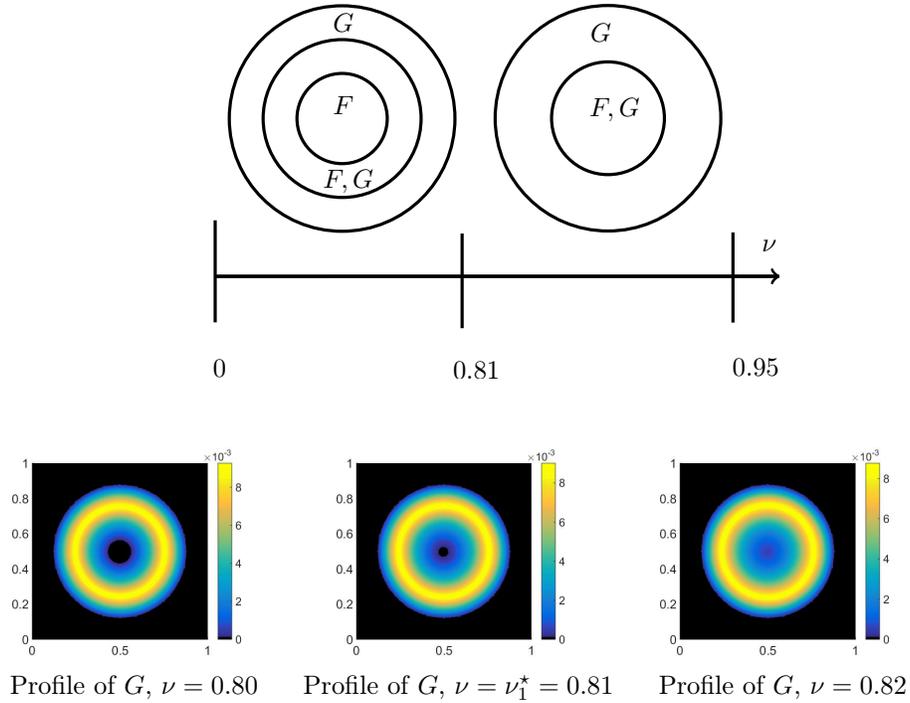
%%%%%%%%%%%%%%%%%%%%%%%%%%%%%%%%%%%%%%%
%%%%%%%%%%%%%%%%%%%%%%%%%%%%%%%%%%%%%

\begin{figure}[!htbp]
\centering 
\begin{subfigure}{0.55\textwidth}

\begin{tikzpicture}[scale=1.5][line cap=round,line join=round,>=triangle 45,x=1.0cm,y=1.0cm]
\clip(4.75,-0.9) rectangle (10.1,2.5);
\draw [line width=1.2pt] (2.1915382332828153,0.38)-- (2.1915382332828153,-0.46);
\draw [line width=1.2pt] (4.766114743820066,0.3896312668968353)-- (4.766114743820066,-0.41036873310316474);
\draw (0.9651046470255671,2.5419297438246513) node[anchor=north west] {$G$};
\draw (5.714214185378386,1.584552823059546) node[anchor=north west] {$F,G$};
\draw (5.858126595631501,2.322002735280388) node[anchor=north west] {$F$};
\draw (8.400579176769877,1.6304844788882493) node[anchor=north west] {$G$};
\draw (8.304637569934467,1.01267910430023427) node[anchor=north west] {$F,G$};
\draw (8.388586475915451,2.35179443421157988) node[anchor=north west] {$F$};
\draw (3.243717809366566,2.4459881369892407) node[anchor=north west] {$G$};
\draw (9.503907655377098,0.5271560002810254) node[anchor=north west] {$\nu$};
\draw (7.177323689618395,-0.6001578800350511) node[anchor=north west] {1.10};
\draw (4.69483461275215,-0.6481286834527564) node[anchor=north west] {0.95};
\draw (2.0324550230695086,-0.6481286834527564) node[anchor=north west] {0.81};
\draw (-0.10224572901837446,-0.27635495696553974) node[anchor=north west] {$0$};
\draw [line width=1.2pt] (0.0029784765304113965,0.4967681822044667)-- (0.0029784765304113965,-0.4083766692161724);
\draw [->,line width=1.2pt] (0.0029784765304113965,0.) -- (10.,0.);
\draw [line width=1.2pt] (1.1280257214749543,1.3995659603922253) circle (1.cm);
\draw [line width=1.2pt] (1.1280257214749543,1.3995659603922253) circle (0.7cm);
\draw [line width=1.2pt] (1.1280257214749543,1.3995659603922253) circle (0.4cm);
\draw [line width=1.2pt] (3.484473766658342,1.4021217888328459) circle (1.cm);
\draw [line width=1.2pt] (3.484473766658342,1.4021217888328459) circle (0.5cm);
\draw [line width=1.2pt] (6.027683444583188,1.38887572729202) circle (0.5cm);
\draw [line width=1.2pt] (6.027683444583188,1.38887572729202) circle (1.cm);
\draw [line width=1.2pt] (8.559347958987894,1.3641887932766428) circle (0.4cm);
\draw [line width=1.2pt] (8.559347958987894,1.3641887932766428) circle (0.7cm);
\draw [line width=1.2pt] (8.559347958987894,1.3641887932766428) circle (1.cm);
\draw [line width=1.2pt] (7.279860913299552,0.4058946826221381)-- (7.279860913299552,-0.3771857952123998);
\end{tikzpicture}
\end{subfigure}

\vspace{0.9cm}

\begin{tabular}{ccc}
\includegraphics[width=0.3\textwidth]{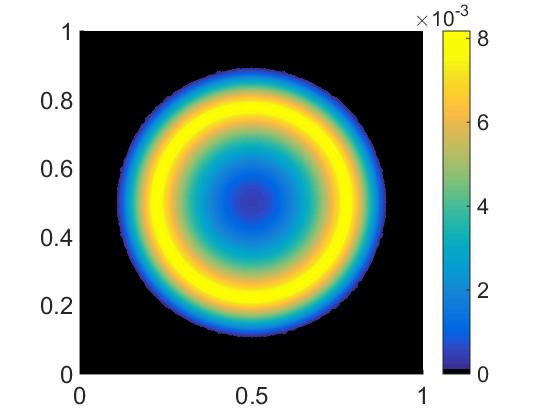} &
\includegraphics[width=0.3\textwidth]{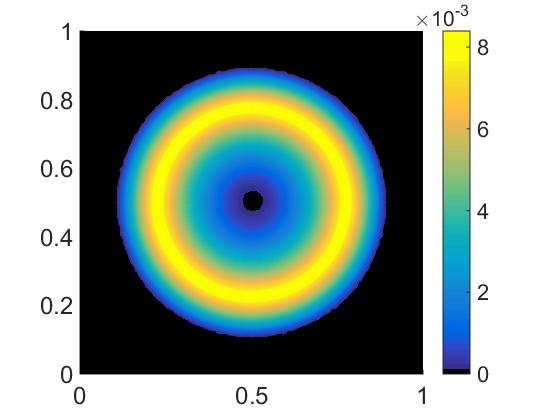}& 
\includegraphics[width=0.3\textwidth]{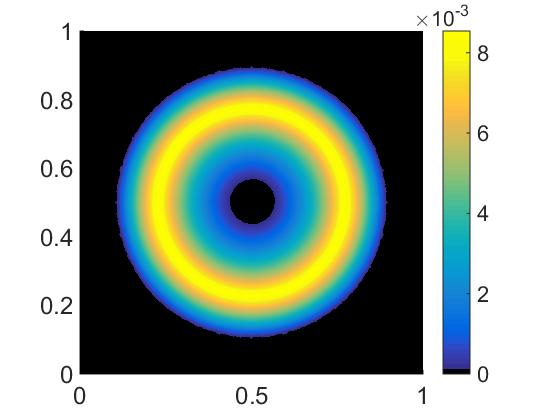} \\
	{Profil of $F$, $\nu=1.09$}&{Profil of $F$,  $\nu=\nu_3^\star=1.10$}&{Profil of $F$, $\nu=1.11$}\\

\end{tabular}
\caption{Topological change of  $E_F$ near the critical value $\nu_3^\star=1.10$}

\end{figure}

%%%%%%%%%%%%%%%%%%%%%%%%%%%%%%%%%%%%%%%%%%%%%%%
%%%%%%%%%%%%%%%%%%%%%%%%%%%%%%%%%%%%%%%%%%%%%
\begin{figure}[!htbp]
\begin{center}
\begin{subfigure}{0.45\textwidth}
\includegraphics[width=0.9\linewidth]{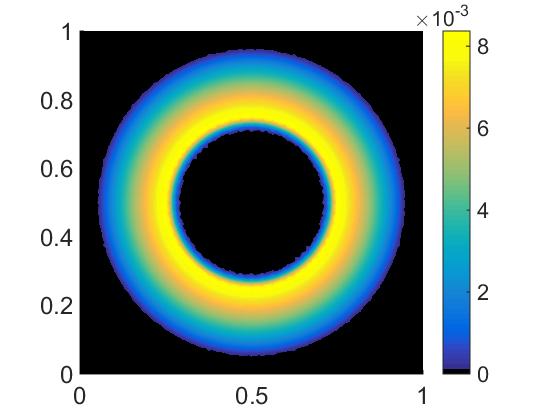} 
\caption{Profile of $F$}

\end{subfigure}
\begin{subfigure}{0.45\textwidth}
\includegraphics[width=0.9\linewidth]{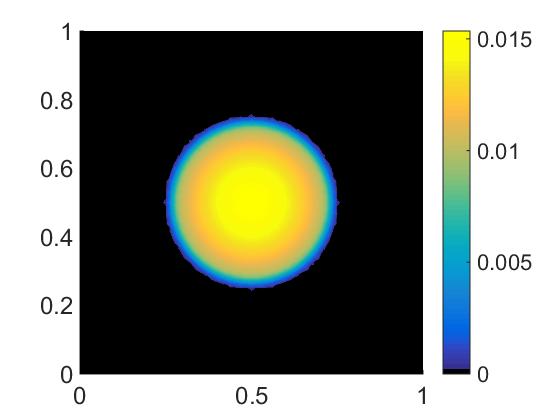}
\caption{Profile of $G$}

\end{subfigure}
 \begin{subfigure}{0.45\textwidth}
\includegraphics[width=0.9\linewidth]{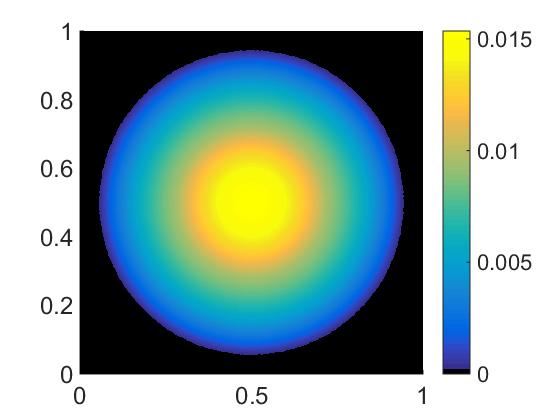} 
\caption{Profile of $F+G$}

\end{subfigure}
\begin{subfigure}{0.45\textwidth}
\centering
\begin{tikzpicture}[scale=0.8][line cap=round,line join=round,>=triangle 45,x=1.0cm,y=1.0cm]
\clip(6.,-2.5) rectangle (10.1,1.8);
\draw [line width=1.2pt] (-1.04,-0.44) circle (0.8590692637965812cm);
\draw [line width=1.2pt] (-1.04,-0.44) circle (1.5153877391611692cm);
\draw [line width=1.2pt] (3.02,-0.42) circle (0.8955445270895244cm);
\draw [line width=1.2pt] (3.02,-0.42) circle (1.4756693396557374cm);
\draw [line width=1.2pt] (8.02,-0.46) circle (0.8357032966310483cm);
\draw [line width=1.2pt] (8.02,-0.46) circle (1.360147050873545cm);
\draw [line width=1.2pt] (8.02,-0.46) circle (1.9054658223122252cm);
\draw [line width=1.2pt] (13.02,-0.56) circle (0.7727871634544661cm);
\draw [line width=1.2pt] (13.02,-0.56) circle (1.366162508635046cm);
\draw [line width=1.2pt] (13.02,-0.56) circle (1.9499743588057767cm);
\draw (7.794094665664915,-0.10265828836828104) node[anchor=north west] {$G$};
\draw (7.538927668875078,-1.280838740523188) node[anchor=north west] {$F,G$};
\draw (7.79672426746807,1.45473328324567944) node[anchor=north west] {$F$};
\end{tikzpicture}
\caption{The fourth configuration}
\end{subfigure}
\caption{Self-similar profiles for $\nu=2$}
\label{fig:cas_mu=2}
\end{center}
\end{figure}

\begin{figure}[!htbp]
\centering
\begin{subfigure}{0.3\textwidth}
\includegraphics[width=0.95\linewidth]{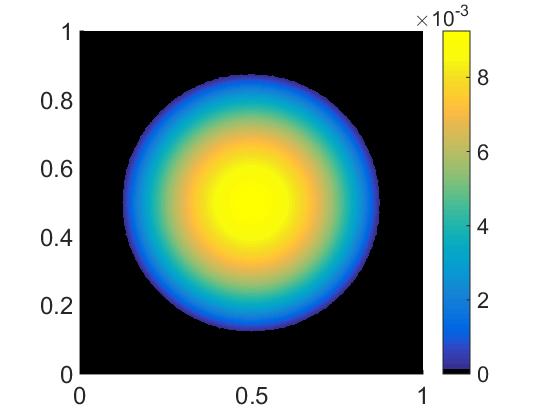} 
\caption{Profile of $F$}

\end{subfigure}
\begin{subfigure}{0.3\textwidth}
\includegraphics[width=0.95\linewidth]{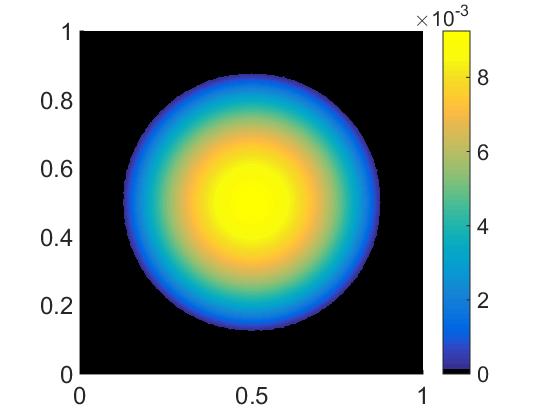}
\caption{Profile of $G$}

\end{subfigure}
 \begin{subfigure}{0.3\textwidth}
\includegraphics[width=0.95\linewidth]{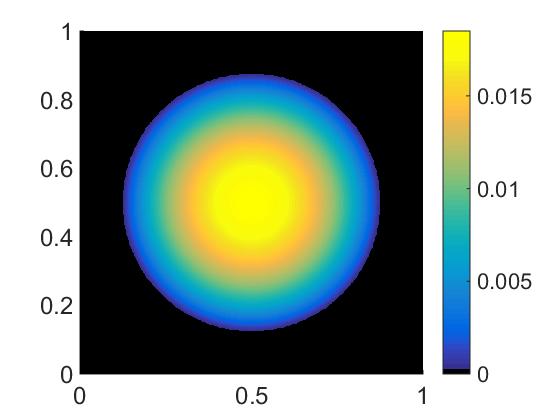} 
\caption{Profile of $F+G$}

\end{subfigure}
\caption{The critical case $\nu = \nu_2^\star$ where $E_F = E_G$.}
\label{fig:cas_mu=0.85}
\end{figure}

%%%%%%%%%%%%%%%%%%%%%%%%%%%%%%%%%%%%%%%%%%%%%
\begin{figure}[!htbp]
\centering
\begin{subfigure}{0.45\textwidth}
\includegraphics[width=0.9\linewidth]{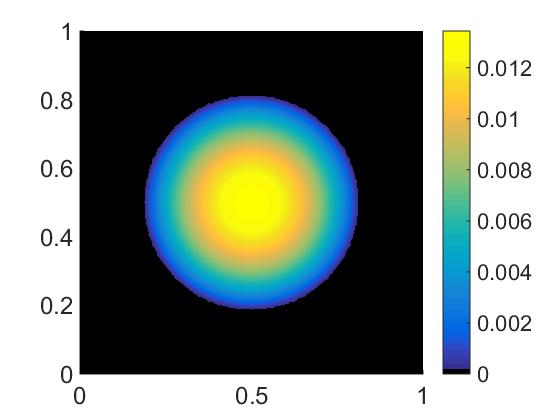} 
\caption{Profile of $F$ for $\nu = \rho$}

\end{subfigure}
\begin{subfigure}{0.45\textwidth}
\includegraphics[width=0.9\linewidth]{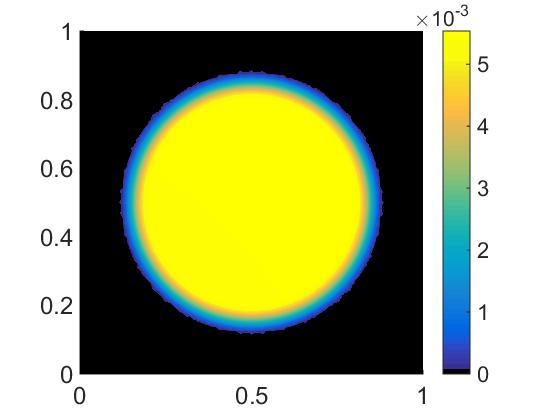}
\caption{Profile of $G$ for $\nu = \rho$}

\end{subfigure}
 \begin{subfigure}{0.45\textwidth}
\includegraphics[width=0.9\linewidth]{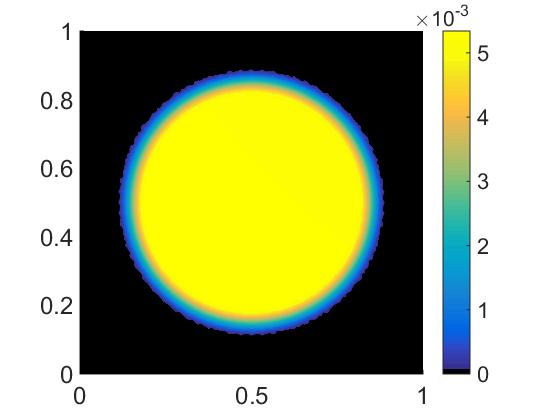} 
\caption{Profile of $F$ for $\nu = 1$}

\end{subfigure}
 \begin{subfigure}{0.45\textwidth}
\includegraphics[width=0.9\linewidth]{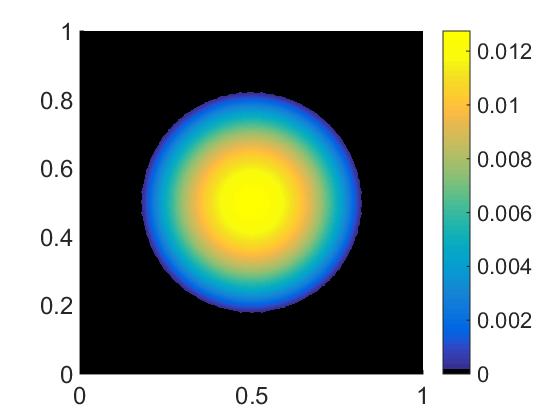} 
\caption{Profile of $G$ for $\nu = 1$}
\end{subfigure}
\caption{The critical values $\nu = \rho$ and $\nu = 1$ where the concavity changes on $E_F \cap E_G$, 
see~\eqref{OnE_FandE_G}.}
\end{figure}

%%%%%%%%%%%%%%%%%%%%%%%%%%%%%%%%%%%%%%%%%%%%%%%%%%%%%%%
%%%%%%%%%%%%%%%%%%%%%%%%%%%%%%%%%%%%%%%%%%%%%%%%%%%%%%%

\begin{figure}[!htbp]
\begin{center}
\begin{tabular}{cc}
\includegraphics[width=0.35\linewidth]{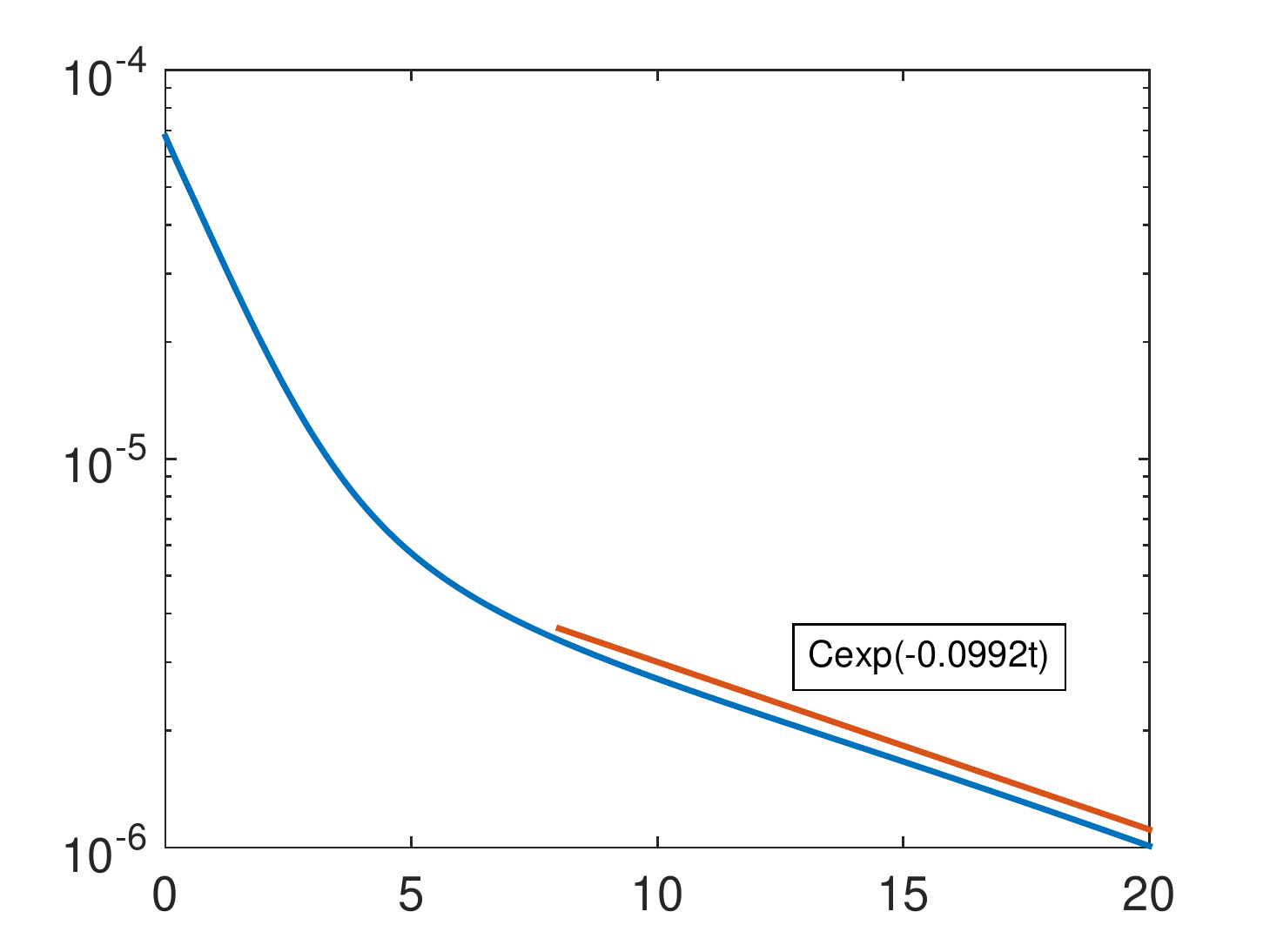} &
\includegraphics[width=0.35\linewidth]{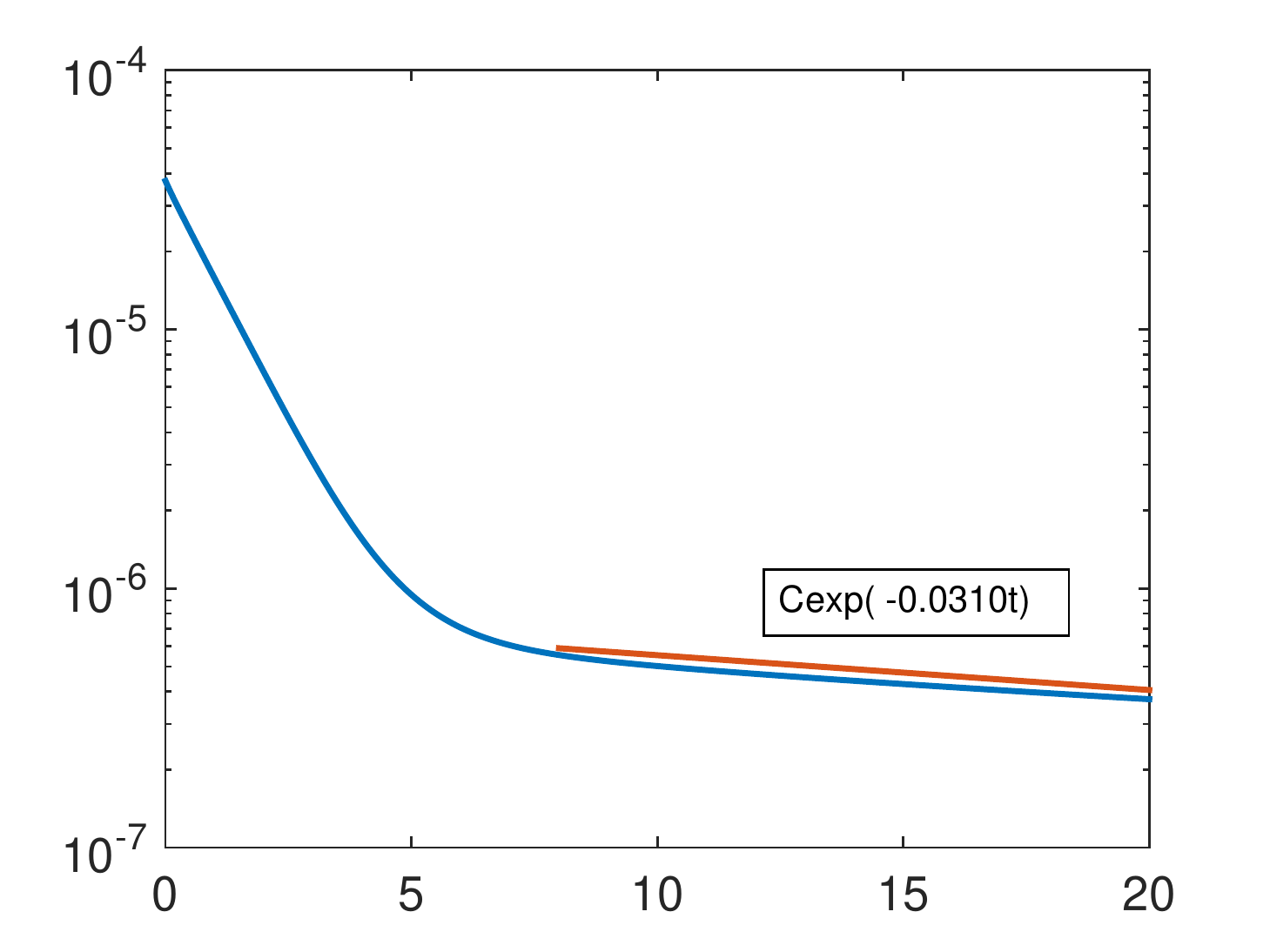}\\
{$\nu=0.40$}&{$\nu=0.90$}\\
\includegraphics[width=0.35\linewidth]{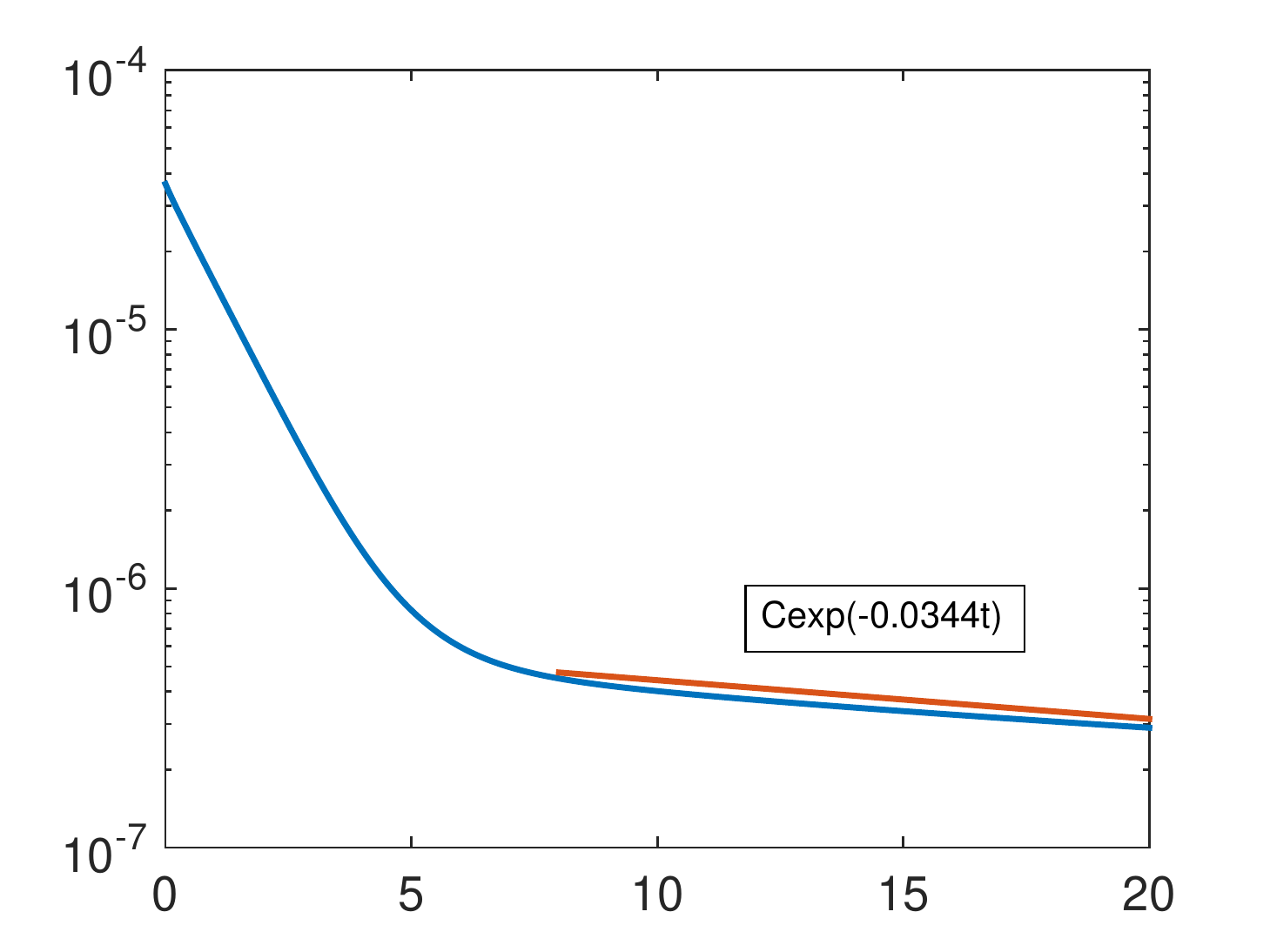} &
\includegraphics[width=0.35\linewidth]{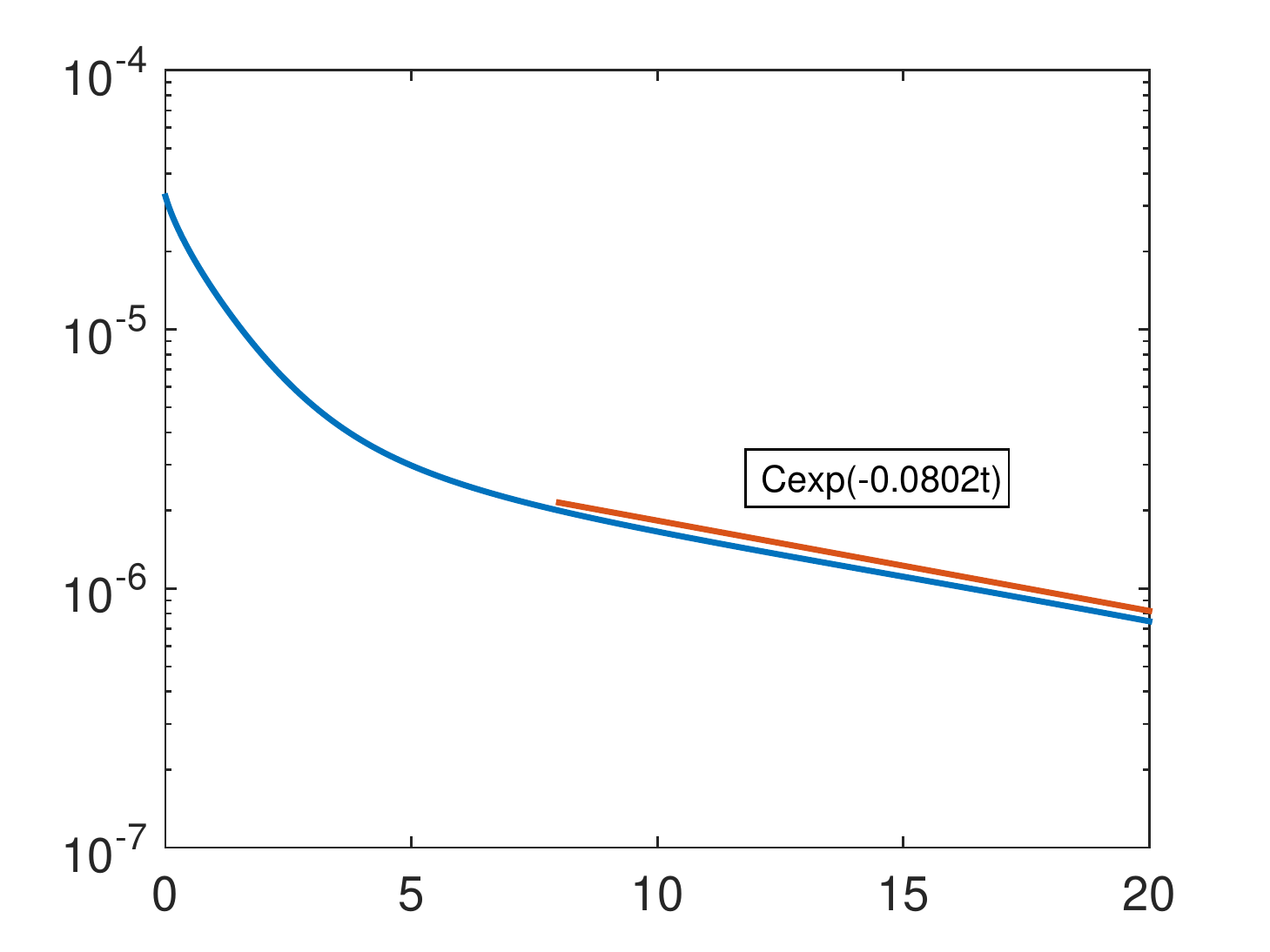}\\
{$\nu=0.95$}&{$\nu=2$}\\
\end{tabular}
\caption{The rate of convergence of the discrete relative energy}
\label{rate_conv_rel}
\end{center}
\end{figure}

Figure~\ref{rate_conv_rel} suggests that the convergence of the discrete solution to the scheme towards the discrete equilibrium occurs at exponential rate. More precisely we have
\be \Ee(f,g|F,G)\leq C\exp (-p(\nu)t),\label{p_rate}\ee
where the rate $p(\nu)$ strongly depends on $\nu$. We plot on Figure~\ref{evolution_rate} the function $\nu \mapsto p(\nu)$ obtained experimentally. At its minimum, the function $p$ is close to 0. This prohibits to conclude to the exponential convergence whatever $\nu \in (0,+\infty)$ and whatever the initial data for the continuous model.
%%%%%%%%%%%%%%%%%%%%%%%%%%%%%%%%%%%%%%%%%%%%%%%%%%%%%%%%%%%
% figure vitesse de convergence en fonction de mu
%%%%%%%%%%%%%%%%%%%%%%%%%%%%%%%%%%%%%%%%%%%%%%%
\begin{figure}[htbp]
\begin{center}
\includegraphics[scale=0.8]{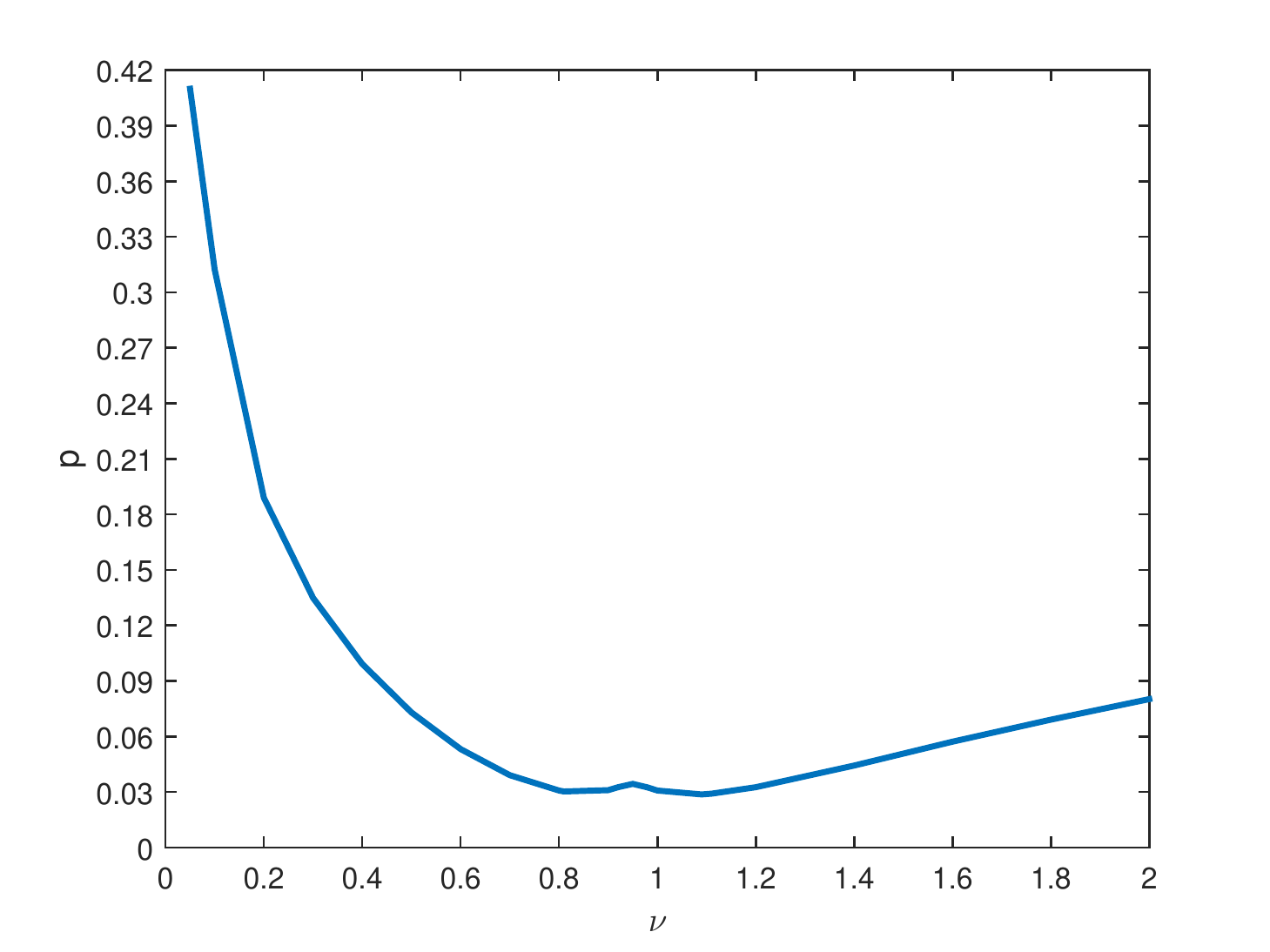}
\caption{The rate of convergence $p$ in (\ref{p_rate}) following  the values of $\nu$}
\label{evolution_rate}
\end{center}
\end{figure}

%%%%%%%%%%%%%%%%%%%%%%%%%%%%%%%%%%%%%%%%%%%%%%%%%%%%
%               Appendix
%%%%%%%%%%%%%%%%%%%%%%%%%%%%%%%%%%%%%%%%%%%%%
\appendix
\section{auxiliary results}

\begin{lem}
The embedding $L^1(\R^2)\cap L^2(\R^2)$ in $(W^{1,4}(\R^2))'$ is continuous. \label{embedding_L1L2_dual_W14}
\end{lem}
\begin{proof}
For $g\in L^1\cap L^2$ and $h\in W^{1,4}$ , one has
\begin{align*}
\left|  <g, h>\right |&\leq \int_{\{|x|\leq 1\}} |g||h|\diff x+ \int_{\{|x|>1\}} |g||h|\diff x \\ & \leq \pi^{1/4}\|g\|_{L^2}\|h\|_{L^4}+\|g\|_{L^1}\|h\|_{L^\infty}\\ &\leq C\left(\|g\|_{L^2}+\|g\|_{L^1} \right) \|h\|_{W^{1,4}},
\end{align*}
thanks to the continuous embedding  $W^{1,4}(\R^2)$ in $L^4(\R^2)\cap L^\infty (\R^2).$
\end{proof}
%%%%%%%%%%%%%%%%%%%%%%%%%%%%%%%%%%%%%

\bibliography{biblio}

\begin{thebibliography}{10}

\bibitem{moifvca2}
A.~Ait Hammou~Oulhaj.
\newblock A finite volume scheme for a seawater intrusion model with
  cross-diffusion.
\newblock In {\em Finite volumes for complex applications {VIII}---methods and
  theoretical aspects}, volume 199 of {\em Springer Proc. Math. Stat.}, pages
  421--429. Springer, Cham, 2017.

\bibitem{intrusion}
A.~Ait Hammou~Oulhaj.
\newblock Numerical analysis of a finite volume scheme for a seawater intrusion
  model with cross-diffusion in an unconfined aquifer.
\newblock {\em Numer. Methods Partial Differential Equations}, 2017.
\newblock to appear, DOI: 10.1002/num.22234.

\bibitem{Ambrosio2008}
L.~Ambrosio, N.~Gigli, and G.~Savar\'e.
\newblock {\em Gradient flows in metric spaces and in the space of probability
  measures}.
\newblock Lectures in Mathematics ETH Z\"urich. Birkh\"auser Verlag, Basel,
  second edition, 2008.

\bibitem{Arkeryd72}
L.~Arkeryd.
\newblock On the {B}oltzmann equation. {I}: Existence.
\newblock {\em Arch. Ration. Mech. Anal.}, 45(1):1--16, 1972.

\bibitem{AMTU01}
A.~Arnold, P.~Markowich, G.~Toscani, and A.~Unterreiter.
\newblock On convex {S}obolev inequalities and the rate of convergence to
  equilibrium for {F}okker-{P}lanck type equations.
\newblock {\em Comm. Partial Differential Equations}, 26(1-2):43--100, 2001.

\bibitem{Bolley2012}
F.~Bolley, I.~Gentil, and A.~Guillin.
\newblock Convergence to equilibrium in {W}asserstein distance for
  {F}okker-{P}lanck equations.
\newblock {\em J. Funct. Anal.}, 263(8):2430--2457, 2012.

\bibitem{Bolley2013}
F.~Bolley, I.~Gentil, and A.~Guillin.
\newblock Uniform convergence to equilibrium for granular media.
\newblock {\em Arch. Ration. Mech. Anal.}, 208(2):429--445, 2013.

\bibitem{Carlen2005}
E.~A. Carlen and S.~Ulusoy.
\newblock An entropy dissipation-entropy estimate for a thin film type
  equation.
\newblock {\em Commun. Math. Sci.}, 3(2):171--178, 2005.

\bibitem{Carillo_Toscani1998}
J.~A. Carrillo and G.~Toscani.
\newblock Exponential convergence toward equilibrium for homogeneous
  {F}okker-{P}lanck-type equations.
\newblock {\em Math. Methods Appl. Sci.}, 21(13):1269--1286, 1998.

\bibitem{Carillo_Toscani2000}
J.~A. Carrillo and G.~Toscani.
\newblock Asymptotic {$L^1$}-decay of solutions of the porous medium equation
  to self-similarity.
\newblock {\em Indiana Univ. Math. J.}, 49(1):113--142, 2000.

\bibitem{Li}
C.~Choquet, J.~Li, and C.~Rosier.
\newblock Global existence for seawater intrusion models: comparison between
  sharp interface and sharp-diffuse interface approaches.
\newblock {\em Electron. J. Differential Equations}, pages No. 126, 27, 2015.

\bibitem{Desvillettes2006}
L.~Desvillettes and K.~Fellner.
\newblock Exponential decay toward equilibrium via entropy methods for
  reaction-diffusion equations.
\newblock {\em J. Math. Anal. Appl.}, 319(1):157--176, 2006.

\bibitem{Desvillettes2015}
L.~Desvillettes and K.~Fellner.
\newblock Duality and entropy methods for reversible reaction-diffusion
  equations with degenerate diffusion.
\newblock {\em Math. Methods Appl. Sci.}, 38(16):3432--3443, 2015.

\bibitem{Escher3}
J.~Escher, {\relax Ph}.~Lauren\c{c}ot, and B.-V. Matioc.
\newblock Existence and stability of weak solutions for a degenerate parabolic
  system modelling two-phase flows in porous media.
\newblock {\em Ann. Inst. H. Poincar\'e Anal. Non Lin\'eaire.}, 28(4):583--598,
  2011.

\bibitem{Escher}
J.~Escher, A.-V. Matioc, and B.-V. Matioc.
\newblock Modelling and analysis of the {M}uskat problem for thin fluid layers.
\newblock {\em J. Math. Fluid Mech.}, 14(2):267--277, 2012.

\bibitem{Escher2}
J.~Escher and B.-V. Matioc.
\newblock Existence and stability of solutions for a strongly coupled system
  modelling thin fluid films.
\newblock {\em NoDEA Nonlinear Differential Equations Appl.}, 20(3):539--555,
  2013.

\bibitem{Eymard3}
R.~Eymard, T.~Gallou{\"e}t, and R.~Herbin.
\newblock Finite volume methods.
\newblock In {\em Handbook of numerical analysis, {V}ol. {VII}}, Handb. Numer.
  Anal., VII, pages 713--1020. North-Holland, Amsterdam, 2000.

\bibitem{Gajewski1996}
H.~Gajewski and K.~G\"artner.
\newblock On the discretization of van {R}oosbroeck's equations with magnetic
  field.
\newblock {\em Z. Angew. Math. Mech.}, 76(5):247--264, 1996.

\bibitem{Gajewski1986}
H.~Gajewski and K.~Gr\"oger.
\newblock On the basic equations for carrier transport in semiconductors.
\newblock {\em J. Math. Anal. Appl.}, 113(1):12--35, 1986.

\bibitem{Gajewski1989}
H.~Gajewski and K.~Gr\"oger.
\newblock Semiconductor equations for variable mobilities based on {B}oltzmann
  statistics or {F}ermi-{D}irac statistics.
\newblock {\em Math. Nachr.}, 140:7--36, 1989.

\bibitem{Glitzky}
A.~Glitzky.
\newblock Exponential decay of the free energy for discretized
  electro-reaction-diffusion systems.
\newblock {\em Nonlinearity}, 21(9):1989--2009, 2008.

\bibitem{Herbin}
R.~Herbin and F.~Hubert.
\newblock Benchmark on discretization schemes for anisotropic diffusion
  problems on general grids.
\newblock In {\em Finite volumes for complex applications {V}}, pages 659--692.
  ISTE, London, 2008.

\bibitem{Jazar}
M.~Jazar and R.~Monneau.
\newblock Derivation of seawater intrusion models by formal asymptotics.
\newblock {\em SIAM J. Appl. Math.}, 74(4):1152--1173, 2014.

\bibitem{Jungel_entropy2016}
A.~J\"ungel.
\newblock {\em Entropy methods for diffusive partial differential equations}.
\newblock SpringerBriefs in Mathematics. Springer, [Cham], 2016.

\bibitem{laurencot}
{\relax Ph}.~Lauren\c{c}ot and B.-V. Matioc.
\newblock A gradient flow approach to a thin film approximation of the {M}uskat
  problem.
\newblock {\em Calc. Var. Partial Differential Equations}, 47(1-2):319--341,
  2013.

\bibitem{laurencot2}
{\relax Ph}.~Lauren\c{c}ot and B.-V. Matioc.
\newblock A thin film approximation of the {M}uskat problem with gravity and
  capillary forces.
\newblock {\em J. Math. Soc. Japan}, 66(4):1043--1071, 2014.

\bibitem{Laurencot_CompoLong}
{\relax Ph}.~Lauren\c{c}ot and B.-V. Matioc.
\newblock Self-{S}imilarity in a {T}hin {F}ilm {M}uskat {P}roblem.
\newblock {\em SIAM J. Math. Anal.}, 49(4):2790--2842, 2017.

\bibitem{MMS09}
D.~Matthes, R.~J. McCann, and G.~Savar\'e.
\newblock A family of nonlinear fourth order equations of gradient flow type.
\newblock {\em Comm. Partial Differential Equations}, 34(10-12):1352--1397,
  2009.

\bibitem{Otto2001}
F.~Otto.
\newblock The geometry of dissipative evolution equations: the porous medium
  equation.
\newblock {\em Comm. Partial Differential Equations}, 26(1-2):101--174, 2001.

\bibitem{Santambrogio_OTAM}
F.~Santambrogio.
\newblock {\em Optimal Transport for Applied Mathematicians: Calculus of
  Variations, PDEs, and Modeling}.
\newblock Progress in Nonlinear Differential Equations and Their Applications
  87. Birkh\"auser Basel, 1 edition, 2015.

\bibitem{Sim87}
J.~Simon.
\newblock Compact sets in the space {$L\sp p(0,T;B)$}.
\newblock {\em Ann. Mat. Pura Appl. (4)}, 146:65--96, 1987.

\bibitem{Toscani_Villani2000}
G.~Toscani and C.~Villani.
\newblock On the trend to equilibrium for some dissipative systems with slowly
  increasing a priori bounds.
\newblock {\em J. Statist. Phys.}, 98(5-6):1279--1309, 2000.

\bibitem{Vazquez2007}
J.~L. V\'azquez.
\newblock {\em The porous medium equation}.
\newblock Oxford Mathematical Monographs. The Clarendon Press, Oxford
  University Press, Oxford, 2007.
\newblock Mathematical theory.

\bibitem{WoMa00}
A.~W. Woods and R.~Mason.
\newblock The dynamics of two-layer gravity-driven flows in permeable rock.
\newblock {\em J. Fluid Mech.}, 421:83--114, 2000.

\bibitem{ZM15}
J.~Zinsl and D.~Matthes.
\newblock Exponential convergence to equilibrium in a coupled gradient flow
  system modeling chemotaxis.
\newblock {\em Anal. PDE}, 8(2):425--466, 2015.

\end{thebibliography}
\bibliographystyle{abbrv}
\end{document}